\documentclass[11pt]{article}
\usepackage{amsfonts}
\usepackage{graphicx}
\usepackage{amsmath}
\usepackage{amsthm}
\usepackage[hidelinks]{hyperref}
\usepackage{comment}
\usepackage[numbers,sort]{natbib}
\usepackage{enumerate}
\usepackage{cancel}
\usepackage{esint}

\newtheorem{theorem}{Theorem}[section]

\newtheorem{corollary}[theorem]{Corollary}

\newtheorem{lemma}[theorem]{Lemma}

\newtheorem{proposition}[theorem]{Proposition}

\theoremstyle{definition}
\newtheorem{definition}[theorem]{Definition}
\newtheorem{example}[theorem]{Example}
\newtheorem{notation}[theorem]{Notation}

\theoremstyle{remark}
\newtheorem{remark}[theorem]{Remark}

\numberwithin{equation}{section}

\newcommand{\defword}[1]{\textbf{#1}}

\renewcommand{\Re}{\operatorname{Re}}

\newcommand{\loc}{\mathrm{loc}}
\newcommand{\pp}[1]{\frac{\partial}{\partial #1}}

\newcommand{\spanop}{\operatorname{span}}

\newcommand{\ignorethis}[1]{}

\begin{document}
\author{Nathaniel Eldredge\footnote{
    School of Mathematical Sciences, University of Northern Colorado,
    501 20th St.\ Box 122, Greeley, CO 80639 USA.  Email
    \texttt{neldredge@unco.edu}.
}}
\title{Strong hypercontractivity and strong logarithmic Sobolev inequalities
  for log-subharmonic functions on stratified Lie groups}
\date{\today\ }
\maketitle

\begin{abstract}
  On a stratified Lie group $G$ equipped with hypoelliptic heat kernel
  measure, we study the behavior of the dilation semigroup on $L^p$ spaces of
  log-subharmonic functions.  We consider a notion of
  strong hypercontractivity and a strong logarithmic Sobolev
  inequality, and show that these properties are equivalent for any
  group $G$.  Moreover, if $G$ satisfies a classical logarithmic
  Sobolev inequality, then both properties hold.  This extends similar
  results obtained by Graczyk, Kemp and Loeb in the Euclidean setting.
\end{abstract}

\tableofcontents

\section{Introduction}\label{sec-intro}

\subsection{Background and motivation}\label{background-sec}

The topic of this paper is inspired by two papers of P.~Graczyk,
T.~Kemp, and J.-J.~Loeb \cite{gkl2010,gkl2015}, in which they
introduced notions of strong hypercontractivity and a strong
logarithmic Sobolev inequality for log-subharmonic functions on real
Euclidean space equipped with an appropriate probability measure, and
showed the intrinsic equivalence of these two notions.  In the present
paper, we extend their results to the setting of a stratified real Lie
group equipped with hypoelliptic heat kernel measure, which we view in
this context as a natural generalization of Euclidean space with
Gaussian measure.

As motivation, we begin by recalling the classical notion of
hypercontractivity and its relationship to the logarithmic Sobolev
inequality.  Let $\mu$ be standard Gaussian measure on $\mathbb{R}^n$,
and let $A$ be the self-adjoint Ornstein--Uhlenbeck operator on
$L^2(\mu)$ given by $Af(x) = -\Delta f(x) + x \cdot \nabla f(x)$. (For
this introduction, we will work formally and ignore domain
considerations.)  \emph{Hypercontractivity} is the statement that
\begin{multline}\label{HC-intro}
  \|e^{-tA} f\|_{L^q(\mu)} \le \|f\|_{L^p(\mu)}, \qquad
   t \ge t_N(p,q) := \frac{1}{2}\log\left(\frac{q-1}{p-1}\right), \\ f \in L^p(\mu), \, 1 < p \le q < \infty.
\end{multline}
This result was proved by E.~Nelson
\cite{nelson66,nelson-free-markov} with improvements by J.~Glimm \cite{glimm68}; see \cite{gross-lsi-survey-2006}
for a broad historical survey of results in this area.  Intuitively,
\eqref{HC-intro} says that after a certain characteristic time $t_N$,
known as ``Nelson's time,'', the Ornstein--Uhlenbeck semigroup
$e^{-tA}$ improves integrability from $L^p(\mu)$ to $L^q(\mu)$.  The
value of $t_N$ given in \eqref{HC-intro} is sharp.

In the same context, the \emph{logarithmic Sobolev inequality}, in its
``$L^1$ form,'' is the
statement that
\begin{multline}
  \label{LSI-intro}
  \int f \log f\,d\mu \le \frac{1}{2} \int \frac{|\nabla
    f|^2}{f}\,d\mu + \|f\|_{L^1(\mu)} \log \|f\|_{L^1(\mu)}, \\ f
  \in C^1(\mathbb{R}^n),\, f > 0
\end{multline}
or equivalently, in the perhaps more familiar ``$L^2$ form'',
\begin{multline}
  \label{LSI-intro-L2}
  \int |f|^2 \log |f|\,d\mu \le  \int |\nabla
    f|^2\,d\mu + \|f\|_{L^2(\mu)}^2 \log \|f\|_{L^2(\mu)}, \qquad f
  \in C^1(\mathbb{R}^n)
\end{multline}
where the equivalence follows by replacing $f$ by $|f|^2$ or vice
versa.  The earliest known version of this inequality is due to
A.J.~Stam \cite{stam59}, with another version discovered independently
by P.~Federbush \cite{federbush69}.  The form given here was obtained
by L.~Gross \cite{gross75}, who coined the name.  Gross also showed,
at the level of Markovian semigroups, that \eqref{LSI-intro} and
\eqref{HC-intro} are equivalent.  For instance, \eqref{LSI-intro-L2}
can be formally obtained from \eqref{HC-intro} by setting $p=2$, $q =
1+e^{2t}$, so that $t(p,q)=t$, and differentiating at $t=0$.

In 1983, S.~Janson \cite{janson-hypercontractivity-1983} discovered a
fascinating phenomenon of hypercontractivity in a complex setting.
Consider \eqref{HC-intro} with $n=2$ and identify $\mathbb{R}^2$ with
$\mathbb{C}$.  Janson showed that if we restrict the inequality
\eqref{HC-intro} to the space $\mathcal{H}$ of holomorphic functions, then we obtain the
following improvement:
\begin{multline}\label{complex-sHC-intro}
  \|e^{-tA} f\|_{L^q(\mu)} \le \|f\|_{L^p(\mu)}, \qquad 
   t \ge t_J(p,q) :=
  \frac{1}{2} \log \left(\frac{q}{p}\right), \\ f \in \mathcal{H}
  \cap L^p(\mu), \, 0 < p \le q < \infty.
\end{multline}
In this result, the critical time $t_J := \frac{1}{2}
\log\left(\frac{p}{q}\right)$ (``Janson's time'') is strictly smaller
than Nelson's time, so integrability improves faster if the initial
function $f$ is holomorphic.  Moreover, Janson's result has content
even if $p=1$ or $0 < p < 1$.  Inequalities of this form have come to
be called \emph{(complex) strong hypercontractivity}.  For alternate
proofs, extensions (including to $\mathbb{C}^n$), and related results,
see
\cite{carlen-integral-identities,janson-complex-hypercontractivity-1997,zhou-contractivity-1991}.

Part of the reason for this strengthening in the holomorphic case is
that, since holomorphic functions are harmonic, the action of $A$ on
holomorphic functions reduces to that of the first-order operator
$Ef(x) = x \cdot \nabla f(x)$, which is simply the generator of
dilations on $\mathbb{C}^n = \mathbb{R}^{2n}$.  This idea was pursued
by Graczyk, Kemp, and Loeb in \cite{gkl2010,gkl2015}, in which they
chose to explicitly consider the behavior of the dilation semigroup
$e^{-tE}$ on real Euclidean space $\mathbb{R}^n$.  In this setting,
the holomorphic functions are replaced by the log-subharmonic (LSH)
functions; i.e. those nonnegative functions $f$ for which $\log f$ is
subharmonic.  (This is effectively a generalization: when $n$ is even
and $f$ is holomorphic, then $|f|$ is log-subharmonic.)  In the case
of Gaussian measure $\mu$, they proved the following version of strong
hypercontractivity:
\begin{multline}\label{sHC-intro}
      \|e^{-tE} f\|_{L^q(\mu)} \le \|f\|_{L^p(\mu)}, \qquad
    t \ge t_J(p,q), \\ f \in LSH \cap L^p(\mu), \,
    0 < p \le q < \infty.
\end{multline}
They also obtained a corresponding \emph{strong logarithmic
  Sobolev inequality}:
\begin{equation}\label{sLSI-intro}
  \int f \log f\,d\mu \le \frac{1}{2} \int Ef\,d\mu +
  \|f\|_{L^1(\mu)} \log \|f\|_{L^1(\mu)}, \qquad  f \in LSH.
\end{equation}
The classical logarithmic Sobolev inequality \eqref{LSI-intro} is a
key ingredient in their proof; indeed, \eqref{LSI-intro} implies
\eqref{sLSI-intro} rather directly.  More generally, Graczyk, Kemp and
Loeb proved, for a wider class of measures $\mu$, that the statements
\eqref{sHC-intro} and \eqref{sLSI-intro}\footnote{Here, and for the
  rest of this section, the inequalities stated above should be read
  as including appropriate constants in the obvious places.} are
equivalent.  For instance, as with \eqref{HC-intro} and
\eqref{LSI-intro}, one may formally obtain \eqref{sLSI-intro} from
\eqref{sHC-intro} by taking $p=1$, $q = e^{2t}$ and differentiating at
$t=0$.

In some cases, the hypothesis $f \in LSH \cap L^p(\mu)$ in
\eqref{sHC-intro} must be strengthened to $f \in LSH \cap L^q(\mu)$; they call
this statement \emph{partial strong hypercontractivity}.  We discuss
this subtle issue in Remark \ref{why-Lq}, later in this paper.

Another line of research inspired by Janson's strong
hypercontractivity \eqref{complex-sHC-intro} was to study the
phenomenon in non-Euclidean settings.  In the papers
\cite{gross-hypercontractivity-complex,gross-strong-hypercontractivity},
L.~Gross considered the case of a complex Riemannian manifold $M$
equipped with an arbitrary smooth probability measure $\mu$, where the
Ornstein--Uhlenbeck operator $A$ is taken to be the generator of the
Dirichlet form $\mathcal{E}(f) = \int_M |\nabla f|^2\,d\mu$.  Gross
showed, under certain assumptions, that if $(M,\mu)$ satisfies the
logarithmic Sobolev inequality \eqref{LSI-intro} then it satisfies the
strong hypercontractivity property \eqref{complex-sHC-intro}.  In this
context, it still happens that $A$ reduces, on holomorphic functions,
to a first-order vector field, whose geometric and complex-analytic
properties become crucially important.

In the paper \cite{eldredge-gross-saloff-coste-dila}, L.~Gross,
L.~Saloff-Coste and the present author were interested in extending
the complex Riemannian results of
\cite{gross-hypercontractivity-complex,gross-strong-hypercontractivity}
into a complex sub-Riemannian setting.  We replaced the complex
Riemannian manifold $M$ with a stratified complex  Lie group $G$
equipped with a left-invariant sub-Riemannian geometry, taking the
measure $\mu$ to be the hypoelliptic heat kernel associated to this
geometry.  A relevant feature of stratified Lie groups is that, like
Euclidean space, they
admit a canonical group of dilations.  In this setting, the
Ornstein--Uhlenbeck operator $A$ fails to be holomorphic, so we
studied instead its $L^2(\mu)$-orthogonal projection $B$ onto the
holomorphic functions, which, we showed, coincides with the vector
field $E$ generating the dilations.  We were able to show that, if the
logarithmic Sobolev inequality \eqref{LSI-intro} holds, then complex
strong hypercontractivity \eqref{complex-sHC-intro} holds with the
operator $B$ in place of $A$.  Of course, in retrospect, this
statement is really \eqref{sHC-intro} for holomorphic functions.

We remark in passing that the logarithmic Sobolev inequality
\eqref{LSI-intro} is known to hold for a few stratified complex Lie groups
(specifically, the complex Heisenberg--Weyl groups), but it is not
currently known whether it holds for all of them.

The aim of the present paper is, in a sense, to unify
\cite{eldredge-gross-saloff-coste-dila} with \cite{gkl2010,gkl2015} by
considering statements akin to \eqref{sHC-intro} (in its ``partial''
form) and \eqref{sLSI-intro}, in the setting of a \emph{real}
stratified Lie group $G$, again equipped with a left-invariant
sub-Riemannian geometry and the associated hypoelliptic heat kernel
measure.  Our main Theorem \ref{main-combined} is, roughly,
that \eqref{sHC-intro} and \eqref{sLSI-intro} are equivalent in any
stratified Lie group $G$, and if $G$ satisfies the logarithmic Sobolev
inequality \eqref{LSI-intro} then \eqref{sHC-intro} and
\eqref{sLSI-intro} are both true.  Again, we stress that
\eqref{LSI-intro} is known to hold for some stratified Lie groups
(specifically, the H-type groups), but it is not currently known
whether it holds for all of them; see Remark \ref{open-lsi-remark}
below.

In our view, stratified Lie groups are a natural setting in which to
generalize \eqref{sHC-intro}, \eqref{sLSI-intro}, since the dilation
structure of a stratified Lie group is perhaps the most direct
generalization of the dilation structure of Euclidean space.  Rather
than considering more general measures $\mu$ as in
\cite{gkl2010,gkl2015}, we have chosen to restrict our attention to
the canonical hypoelliptic heat kernel measure: partly because it is
the natural generalization of Gaussian measure in this setting, and
partly because we need to make use of strong heat kernel estimates
from the literature (Theorem \ref{heat-kernel-upper-lower} below).

\subsection{Statement of results}

We briefly summarize the notation required to state our results.
Complete definitions are given in Sections \ref{sec:groups} and
\ref{sec:LSH} below.

Let $G$ be a stratified Lie group equipped with a left-invariant
sub-Riemannian metric $\langle \cdot, \cdot \rangle$, for which the
horizontal space is given by the first layer of the stratification of
the Lie algebra of $G$.  Let $m$ be some normalization of Haar
(Lebesgue) measure on $G$.  We denote by $\nabla$ and $\Delta$ the
canonical sub-gradient and sub-Laplacian induced by the metric, and by
$\rho_s$ the hypoelliptic heat kernel for $\Delta$ at time $s > 0$.  A
function $f \in C^2(G)$ is said to be \emph{log-subharmonic} (LSH) if
$f > 0$ and $\Delta \log f \ge 0$ (we discuss alternative formulations
in Section \ref{weak-lsh}).

We denote by $E$ the vector field which generates the canonical
dilations $\delta_r$ of the group $G$, and by $e^{-tE} f = f \circ
\delta_{e^{-t}}$ the corresponding operator semigroup.

The $L^p$ norms and spaces in the following statements are taken with
respect to the heat kernel probability measure $\rho_s \,dm$ at some
fixed time $s$.  We let $L^{p+} = \bigcup_{q > p} L^q$, and write $f
\in W^{1, p+}$ if $f, |\nabla f| \in L^{p+}$, where $|\cdot|$ is the
norm induced by the metric $\langle \cdot, \cdot \rangle$.

The aim of this paper is to study the relationship between the
following three statements, for fixed constants $c, \beta \ge 0$.  The
time parameter $s > 0$ may be taken as arbitrary; each of the
following statements holds for one $s >0$ iff it holds for all $s>0$,
with the same constants $c,\beta$ (see Remark \ref{scaling-remark}
below).
\begin{itemize}
  \item The \defword{classical logarithmic Sobolev inequality}: 
\begin{multline} \label{LSI}  \tag{LSI}
  \int_G f\log f\,\rho_s\,dm \le \frac{cs}{2} \int_G \frac{|\nabla
    f|^2}{f} \,\rho_s\,dm + \|f\|_{L^1} \log \|f\|_{L^1} + \beta
  \|f\|_{L^1}, \\
  f \in C^1(G), f \ge 0
\end{multline}
We have stated this in its ``$L^1$ form''.  By replacing $f$ by $f^2$,
one can see that \eqref{LSI} is equivalent to the ``$L^2$ form'':
\begin{multline}\label{LSI-L2} \tag{$L^2$-LSI}
  \int_G f^2\,\log |f|\,\rho_s\,dm \le cs \int_G |\nabla
  f|^2\,\rho_s\,dm + \|f\|_{L^2}^2 \log \|f\|_{L^2} + \frac{\beta}{2}
  \|f\|_{L^2}^2, \\
  f \in C^1(G)
\end{multline}
The original formulation \eqref{LSI-intro} of the logarithmic Sobolev inequality
corresponds to taking $\beta = 0$.  When $\beta > 0$, \eqref{LSI} is
sometimes referred to as a ``defective logarithmic Sobolev
inequality''.

As remarked in Section \ref{background-sec}, \eqref{LSI} is well known
to be equivalent to the hypercontractivity of the Ornstein--Uhlenbeck
semigroup of $G$.

\item The \defword{strong logarithmic Sobolev inequality}:
  \begin{multline}\label{sLSI} \tag{sLSI}
    \int_G f \log f\,\rho_s\,dm \le c \int_G Ef\,\rho_s\,dm +
    \|f\|_{L^1} \log \|f\|_{L^1} + \beta \|f\|_{L^1}, \\
    f \in LSH \cap W^{1,1+}
  \end{multline}

  The name ``strong logarithmic Sobolev inequality'' comes from
  \cite{gkl2010}. However, in our present context, we show in Theorem
  \ref{LSI-implies-sLSI} that \eqref{LSI} implies \eqref{sLSI}, so
  \eqref{sLSI} is in fact logically weaker.  Of course, \eqref{sLSI}
  applies to a much smaller class of functions.

\item \defword{(Partial) strong hypercontractivity}:
  \begin{multline}\label{sHC} \tag{sHC}
    \|e^{-tE} f\|_{L^q} \le M(p,q) 
    \|f\|_{L^p}, \qquad 
    t \ge t_J(p,q), \\ f \in LSH \cap L^q, \,
    0 < p \le q < \infty
  \end{multline}
  where
  \begin{equation}\label{M-tJ-def}
    M(p,q) := \exp(\beta \cdot (p^{-1} - q^{-1})), \qquad t_J(p,q) := c \log \left( \frac{q}{p} \right).
  \end{equation}
  Here $t_J(p,q)$ is Janson's time.  Note that the word
  ``contractivity'' is more apt when $\beta = 0$, since in that case,
  $M(p,q) = 1$ and \eqref{sHC} says that $e^{-tE}$ is a contraction
  from a subset of $L^p$ into $L^q$.  The word ``partial'' comes from
  \cite{gkl2015} and refers to the hypothesis $f \in LSH \cap L^q$ (rather than
  $L^p$); see Remark \ref{why-Lq} below.
\end{itemize}

In comparing these statements to \cite{gkl2010,gkl2015}, note that our
$c$ is their $\frac{c}{2}$.

The main results of this paper can be summarized as follows:

\begin{theorem}\label{main-combined}
  In any stratified Lie group $G$, the statements \eqref{sLSI} and
  \eqref{sHC} are equivalent.  If $G$ satisfies \eqref{LSI}, then
  \eqref{sLSI} and \eqref{sHC} are both satisfied.  That is,
  \begin{equation*}
  \text{\eqref{LSI}} \: \Longrightarrow \: \text{\eqref{sLSI}} \:
  \Longleftrightarrow \: \text{\eqref{sHC}}.
  \end{equation*}
\end{theorem}

The implication \eqref{LSI} $\Longrightarrow$ \eqref{sLSI} is Theorem
\ref{LSI-implies-sLSI}; \eqref{sHC} $\Longrightarrow$ \eqref{sLSI} is
Theorem \ref{sHC-implies-sLSI}; and \eqref{sLSI} $\Longrightarrow$
\eqref{sHC} is Theorem \ref{sLSI-implies-sHC}.

\begin{remark}\label{open-lsi-remark}
It is an open problem to determine which stratified Lie
groups satisfy the logarithmic Sobolev inequality \eqref{LSI}, and we
hope this paper may provide additional motivation for further work on
this difficult question. The current state of the art, as far as we
are aware, is that \eqref{LSI} is true for H-type groups
(\cite{eldredge-gradient, li-gradient-h-type}; see Example
\ref{h-type-example} below for definitions and references), and of course in
the ``step 1'' Euclidean case (Example \ref{euclidean-ex}).  In all
other stratified Lie groups, including all those of step $\ge 3$, it
is apparently unknown whether \eqref{LSI} holds or not.
\end{remark}

\begin{corollary}\label{H-type-cor}
  If $G$ is an H-type group, then \eqref{sLSI} and \eqref{sHC} are
  both true.
\end{corollary}

\begin{remark}\label{why-Lq}
  In the statement \eqref{sHC}, the hypothesis $f \in L^q$ may seem
  somewhat unnatural, given that the result is to bound the $L^q$ norm
  by the $L^p$ norm.  It is reasonable to conjecture that if
  \eqref{sHC} holds for all $f \in LSH \cap L^q$ then in fact it holds
  for all $f \in LSH \cap L^p$.  However, the obvious density argument
  is not available, because we do not know in this setting whether
  $LSH \cap L^q(\rho_s)$ is dense in $LSH \cap L^p(\rho_s)$.  (Density
  arguments in stratified Lie groups can be subtle; for example, it is
  shown in \cite[Proposition 8]{lust-piquard-ornstein-uhlenbeck} that
  polynomials are dense in $L^2(\rho_s)$ only for groups of step $m
  \le 4$.)

  This issue arose, in the Euclidean setting, in the work of Graczyk,
  Kemp and Loeb \cite{gkl2015}; our \eqref{sHC} is the analogue of the
  statement ``partial strong hypercontractivity'' appearing in their
  Theorem 1.17.1(b).  The stronger statement, requiring only $f \in
  LSH \cap L^p$ (in our notation), is their Theorem 1.17.1(a); but
  they are able to show this only under significantly stronger
  assumptions on the measure, one of which is that the measure be
  $\alpha$-subhomogeneous for some $\alpha \le 1/c$, where $c$ is our
  constant in the strong logarithmic Sobolev inequality \eqref{sLSI}
  (recall that our $c$ is their $\frac{c}{2}$).  This approach does
  not appear to succeed in our stratified Lie group setting, given
  current technology.  For instance, if we consider the Heisenberg
  group $\mathbb{H}^3$, we can use the heat kernel estimates of
  \cite{li-heatkernel, eldredge-precise-estimates} to see that the
  heat kernel is $\alpha$-subhomogeneous only for $\alpha \ge 2$.
  However, current methods for proving \eqref{LSI} in this setting
  produce a constant $c > \frac{1}{2}$ (see \cite[Sections 1.2 and
    6.1]{bbbc-jfa} in conjunction with \cite[Theorem
    1.6]{driver-melcher}), although we do not know whether this is
  sharp.  The situation for H-type groups is similar
  \cite{eldredge-precise-estimates, li-heat-h-type, eldredge-gradient,
    li-gradient-h-type}.
\end{remark}

\section{Stratified Lie groups and hypoelliptic heat kernels}\label{sec:groups}

In this section, we review the standard definitions and properties of
stratified Lie groups (also known as Carnot groups), and of the
sub-Riemannian geometry and hypoelliptic heat kernels on these groups.
The material in this section is adapted from
\cite{eldredge-gross-saloff-coste-dila}.  Some motivating examples are
discussed in Section \ref{examples-sec}.

\subsection{Stratified Lie groups}
A
comprehensive reference on stratified Lie groups is \cite{blu-book}.

\begin{definition}
  Let $\mathfrak{g}$ be a finite-dimensional  real Lie algebra.  We
  say $\mathfrak{g}$ is \defword{stratified} of step $m$ if it admits a
  direct sum decomposition
  \begin{equation}\label{decomp}
    \mathfrak{g} = \bigoplus_{j=1}^m V_j
  \end{equation}
  and
  \begin{equation*}
    [V_1, V_j] = V_{j+1}, \qquad [V_1, V_m] = 0.
  \end{equation*}

  A finite-dimensional real Lie group $G$ is \defword{stratified} if
  it is connected and simply connected and its Lie algebra
  $\mathfrak{g}$ is stratified.
\end{definition}

Stratified Lie groups are also known as Carnot groups.  Various
equivalent definitions can be found in \cite[Chapters 1 and 2]{blu-book}.

As a trivial example, Euclidean space $\mathbb{R}^n$ with its usual
addition is a (commutative) stratified Lie group of step 1.  (Here the
Lie bracket is simply 0.)  The simplest nontrivial example of a
stratified Lie group is the Heisenberg group $\mathbb{H}_3$, which has
step 2.  See Section \ref{examples-sec} below for further
discussion of these and other examples.

It is easy to check that a stratified Lie group $G$ is necessarily
nilpotent, and thus diffeomorphic to its Lie algebra $\mathfrak{g}$
via the exponential map.  In particular, a stratified Lie group is
diffeomorphic to Euclidean space $\mathbb{R}^n$ as a smooth manifold
(though certainly not isomorphic to $\mathbb{R}^n$ as a Lie group).

\begin{notation}
Let $L_x, R_x : G \to G$ denote the
left and right translation maps $L_x(y)= x  y$ and $R_x(y) = yx$.
\end{notation}

\begin{notation}\label{tilde-hat-notation}
  Let $e$ denote the identity element of $G$.  We identify the Lie
  algebra $\mathfrak{g}$ with the tangent space $T_e G$. For $\xi \in
  \mathfrak{g}$, let $\widetilde{\xi}, \widehat{\xi}$ denote,
  respectively, the unique left-invariant and right-invariant vector
  fields on $G$ with $\widetilde{\xi}(e) = \widehat{\xi}(e) = \xi$.
\end{notation}

\begin{notation}
Being a connected nilpotent Lie group, $G$ is unimodular, so it has a
bi-invariant Haar measure which is unique up to scaling.  For our
purposes, there is no particular natural choice of scaling, so from
now on $m$ will denote some fixed Haar measure on $G$.  Integrals like
$\int_G f(x)\,dx$ will denote Lebesgue integrals with respect to $m$.
It is easy to verify that the Haar measure on $G$ is the push-forward
under the exponential map of Lebesgue measure on $\mathfrak{g}$.
\end{notation}

\begin{notation}
  \textbf{Convolution} on $G$ is defined by
\begin{equation}
  (f\ast g)(x)=\int_{G}f(xy^{-1})g(y)dy=\int_{G}f(z)g(z^{-1}x)dz
\end{equation}
when the integral exists.
\end{notation}

Suppose $\varphi \in C^1_c(G)$, $f \in C^1(G)$ and $\xi \in
\mathfrak{g}$.  By using the formulas $\widetilde{\xi} g(x) =
\frac{d}{dt}|_{t=0} \,g(x e^{t \xi})$ and $\widehat{\xi} g(x) =
\frac{d}{dt}|_{t=0} \,g(e^{t \xi} x)$ and differentiating under the
integral sign, we obtain the identities
\begin{equation}\label{convolution-diff}
  \widetilde{\xi}[\varphi \ast f] = \varphi \ast (\widetilde{\xi} f),
  \qquad \widehat{\xi}[\varphi \ast f] = (\widehat{\xi} \varphi) \ast f.
\end{equation}
\subsection{The dilation semigroup}\label{dilate-subsec}

\begin{definition}
  For $\lambda \ge 0$, the \defword{dilation map}
  $\delta_\lambda$ on
  $\mathfrak{g}$ is defined by
  \begin{equation}
    \delta_{\lambda}(v_{1}+\cdots+v_{m})=\sum_{j=1}^{m}\lambda^{j}v_{j}\qquad
    v_{j}\in V_{j}\quad j=1,\ldots,m.
  \end{equation}
  By an abuse of notation, we will also use $\delta_\lambda$ to denote
  the corresponding map on $G$ defined by $\delta_\lambda(\exp(v)) =
  \exp(\delta_\lambda(v))$.  
\end{definition}

It is straightforward to verify that for each $\lambda > 0$, the
dilation $\delta_\lambda$ on $\mathfrak{g}$ is an automorphism of the
Lie algebra, and the dilation $\delta_\lambda$ on $G$ is an
automorphism of the Lie group.  Also,
\begin{equation}
  \delta_{\lambda\mu}=\delta_{\lambda} \circ \delta_{\mu}\qquad\lambda,\mu
  \ge 0. 
\end{equation}
Moreover,  the derivative at the
identity of $\delta_\lambda : G \to G$ is $(\delta_\lambda)_* =
\delta_\lambda : \mathfrak{g} \to \mathfrak{g}$.

Since $\delta_\lambda$ is a group automorphism, we have the identity
\begin{equation*}
  \delta_\lambda \circ L_x = L_{\delta_\lambda(x)} \circ \delta_\lambda.
\end{equation*}
Hence if $\xi \in \mathfrak{g} = T_e G$ and $\widetilde{\xi}$ is the
corresponding left-invariant vector field, we have
\begin{equation*}
  (\delta_\lambda)_* \widetilde{\xi}(x) = (\delta_\lambda)_* (L_x)_*
  \xi = (\delta_\lambda L_x)_* \xi = (L_{\delta_\lambda(x)}
  \delta_\lambda)_* \xi = (L_{\delta_\lambda}(x))_* (\delta_\lambda)_* \xi.
\end{equation*}
In particular, if $\xi \in V_j$, then $(\delta_\lambda)_* \xi =
\lambda^j \xi$ and so
\begin{equation}\label{xi-dilation}
  (\delta_\lambda)_* \widetilde{\xi}(x) = \lambda^j
  (L_{\delta_\lambda(x)})_* \xi = \lambda^j \widetilde{\xi}(\delta_\lambda(x))
\end{equation}
or in other words
\begin{equation}\label{xi-f-dilation}
  \widetilde{\xi} (f \circ \delta_\lambda) = \lambda^j
  (\widetilde{\xi} f) \circ \delta_{\lambda}.
\end{equation}

The dilation structure is a fundamental property of stratified Lie
groups, and since the aim of this paper is to generalize results on
the dilation in Euclidean space, stratified Lie groups are a natural
setting to consider.  Indeed, in a certain sense, we are studying what
happens if we are allowed to dilate at different rates in different
directions (linearly in $V_1$ directions, quadratically in $V_2$
directions, and so on).  

\begin{definition}
  We define the \defword{dilation vector field} or \defword{Euler
    vector field} $E$ on $G$ by
  \begin{align}
    (Ef)(x) = \frac{d}{dr}\Big|_{r=0}f(\delta_{e^{r}}(x))\qquad f\in C^{1}(G).
  \end{align}
  The main object of study in this paper is the one-parameter semigroup of
  dilation operators
  \begin{equation}
    e^{-tE} f = f \circ \delta_{e^{-t}}, \qquad t \ge 0.
  \end{equation}
\end{definition}

\begin{notation}\label{adapted-coords}
Let $\xi_{j,k}$ be a basis for $\mathfrak{g} = G$ adapted to the
stratification $\{V_j\}$, so that $\{\xi_{j,k} : 1 \le k \le \dim
V_j\}$ is a basis for $V_j$.  Then each $x \in G$ can be written
uniquely as $x = \exp\left(\sum_{j=1}^m \sum_{k=1}^{\dim V_j} x_{j,k}
\xi_{j,k}\right)$, so that $x_{j,k}$ is a smooth system of coordinates
on $G$.
\end{notation}

In this system of coordinates, we have
\begin{equation}\label{E-coords}
  E = \sum_{j=1}^m \sum_{k=1}^{\dim V_j} j x_{j,k} \frac{\partial}{\partial x_{j,k}}.
\end{equation}
In the Euclidean case $m=1$, we have $E = \sum x_k
\frac{\partial}{\partial x_k} = x \cdot \nabla$, a vector field
pointing radially away from the origin with magnitude $|x|$.

\begin{notation}
  The \defword{homogeneous dimension} of $\mathfrak{g}$ or $G$ is
  \begin{equation*}
    D = \sum_{j=1}^m j \dim V_j.
  \end{equation*}
\end{notation}

We note that $\delta_{\lambda}$ scales the Lebesgue measure $m$ by
\begin{equation}\label{m-dilate}
m(\delta_\lambda(A)) = \lambda^{D} m(A).
\end{equation} 
Thus for an integrable function $f$, we have
\begin{equation}\label{m-dilate-integrate}
  \int_G f \circ \delta_\lambda\,dm = \lambda^{-D} \int_G f\,dm.
\end{equation}

To conclude this subsection, we observe that the vector fields discussed
above can be expressed in terms of each other in a well-behaved
manner.

Define the adapted basis $\{\xi_{j,k}\}$ for $\mathfrak{g}$ and the
coordinates $\{x_{j,k}\}$ on $G$ as in Notation
\ref{adapted-coords}.  For each $j,k$, the vector fields
$\pp{x_{j,k}}, \widetilde{\xi_{j,k}}, \widehat{\xi_{j,k}}$ (see
Notation \ref{tilde-hat-notation}) coincide at the identity but in
general nowhere else.

\begin{lemma}\label{poly-vector-fields}
  We can write
  \begin{align}
    \widetilde{\xi_{j,k}} &= \pp{x_{j,k}} + \sum_{\alpha=j+1}^m
    \sum_{\beta=1}^{\dim V_\alpha} a_{j,k}^{\alpha, \beta}
    \pp{x_{\alpha, \beta}} \label{xi-in-x}
    \intertext{and}
    \pp{x_{j,k}} &= \widetilde{\xi_{j,k}} + \sum_{\alpha=j+1}^m
    \sum_{\beta=1}^{\dim V_\alpha} b_{j,k}^{\alpha, \beta}
    \widetilde{\xi_{j,k}} \label{x-in-xi}
  \end{align}
  where the coefficient functions $a_{j,k}^{\alpha, \beta}$,
  $b_{j,k}^{\alpha, \beta}$ are polynomials in the coordinates
  $x_{\alpha, \beta}$.

  We can likewise express $\{\widehat{\xi_{j,k}}\}$ and $\{\pp{x_{j,k}}\}$
  in terms of each other, with polynomial coefficients, as well as
  $\{\widetilde{\xi_{j,k}}\}$ and  $\{\widehat{\xi_{j,k}}\}$.
\end{lemma}

\begin{proof}
  By the Baker--Campbell--Hausdorff formula, in the coordinates
  $\{x_{j,k}\}$, the group operation on $G$ has the form
  \begin{equation*}
    (x y)_{j,k} = x_{j,k} + y_{j,k} + R_{j,k}(x,y)
  \end{equation*}
  where $R_{j,k}$ is a polynomial which only depends on the
  coordinates $x_{\alpha,\beta}, y_{\alpha,\beta}$ with $\alpha < j$.
  See \cite[Proposition 2.2.22 (4)]{blu-book} for details.  Then \eqref{xi-in-x}
  follows immediately, since $\widetilde{\xi_{j,k}}(x)  = \frac{d}{dt}
  \big|_{t=0} x \exp(t \xi_{j,k})$.  We then obtain \eqref{x-in-xi} by
  solving the system \eqref{xi-in-x} for $\pp{x_{j,k}}$.  (Note, for
  instance, that from \eqref{xi-in-x} we have $\widetilde{\xi_{m,k}} =
  \pp{x_{m,k}}$, so that \eqref{x-in-xi} is trivially satisfied when
  $j=m$.  One can then proceed by downward induction on $j$.)

  An identical argument applies to $\{\widehat{\xi_{j,k}}\}$ and
  $\{\pp{x_{j,k}}\}$.  To write $\widetilde{\xi_{j,k}}$ in terms of
  $\{\widehat{\xi_{j,k}}\}$, first write $\{\widehat{\xi_{j,k}}\}$ in
  terms of $\{\pp{x_{j,k}}\}$ as just noted, and then substitute this
  into \eqref{xi-in-x}.
\end{proof}

\begin{corollary}[See also \mbox{\cite[Lemma
        4]{lust-piquard-ornstein-uhlenbeck}}] \label{E-in-xi-cor}
  We can write
  \begin{equation}\label{E-in-xi}
    E = \sum_{j,k} \tilde{c}_{j,k} \widetilde{\xi_{j,k}} = \sum_{j,k}
    \hat{c}_{j,k} \widehat{\xi_{j,k}}
  \end{equation}
  where the coefficient functions $\tilde{c}_{j,k}, \hat{c}_{j,k}$ are
  polynomials in the $x_{\alpha,\beta}$ coordinates.
\end{corollary}

\begin{proof}
  Substitute \eqref{x-in-xi} into \eqref{E-coords}.
\end{proof}

\subsection{Sub-Riemannian geometry on $G$}\label{sub-riemannian-section}

In this subsection, we review some facts about the sub-Riemannian
geometry of a stratified Lie group $G$ and its hypoelliptic
sub-Laplacian.  For background on the general notions of sub-Riemannian
geometry, see \cite{montgomery,rifford,strichartz,strichartz-corrections}.

Fix an inner product $\langle \cdot, \cdot \rangle$ on $V_1 \subset
\mathfrak{g}$.  From now on, when we speak of a stratified Lie group
$G$, we really mean a triple $(G, \langle \cdot, \cdot \rangle, m)$,
including a choice of inner product on $V_1$ and a choice of
normalization for the Haar measure.  Objects such as the
sub-Laplacian, heat kernel, etc, which we discuss below, are not
really intrinsic to the Lie group $G$, but depend on the choice of
$\langle \cdot, \cdot \rangle$ and $m$.

The inner product gives rise to a \defword{sub-Riemannian geometry} on
the smooth manifold $G$ in the following way.  For $x \in G$, let $H_x
= (L_x)_* V_1 \subset T_x G$, so that $H$ is a left-invariant
subbundle of the tangent bundle $TG$.  This $H$ is called the
\defword{horizontal bundle} or \defword{horizontal distribution}.
Then the inner product $\langle \cdot, \cdot \rangle$ on $V_1$
induces, by left translation, an inner product $\langle \cdot, \cdot
\rangle_x$ on $H_x$, which is a left-invariant \defword{sub-Riemannian
  metric} on $G$.  We may drop the subscript $x$ when no confusion
will arise.  Since $V_1$ generates $\mathfrak{g}$, the horizontal bundle
$H$ satisfies H\"ormander's bracket generating condition.  We will use
$|\cdot|$ to denote the norm on $H_x$ induced by $\langle \cdot, \cdot
\rangle$.

This metric gives rise to a canonical left-invariant \defword{sub-Laplacian} $\Delta$ on $G$,
which is easiest to define in terms of a basis.  Let $\xi_1, \dots,
\xi_n$ be an orthonormal basis for $V_1$, and let
\begin{equation}\label{Delta-def}
  \Delta = \widetilde{\xi_1}^2 + \dots + \widetilde{\xi_n}^2
\end{equation}
where, as in Notation \ref{tilde-hat-notation}, $\widetilde{\xi_i}$ is
the extension of $\xi_i$ to a left-invariant vector field on $G$.  It
is easy to check this definition is independent of the basis chosen.
Since $H$ satisfies the bracket generating condition, the operator
$\Delta$ is hypoelliptic \cite{hormander67}.  It is shown in
\cite{purplebook} that $\Delta$, with domain $C^\infty_c(G)$, is
essentially self-adjoint on $L^2(G, m)$.

As a consequence of
\eqref{xi-f-dilation}, we have
\begin{equation}\label{Delta-dilate}
  \Delta [f \circ \delta_\lambda] = \lambda^2 (\Delta f) \circ \delta_\lambda.
\end{equation}
Likewise, if $e^{s \Delta/4}$ is the heat semigroup for $\Delta$, we
have
\begin{equation}\label{semigroup-dilate}
  e^{s \Delta/4} [f \circ \delta_\lambda] = (e^{s \lambda^2
    \Delta/4} f) \circ \delta_\lambda.
\end{equation}

Much more information
about the sub-Laplacian can be found in \cite{blu-book}.

We may also define the \defword{sub-gradient} $\nabla$ by
\begin{equation}
  \nabla f(x) = \sum_{i=1}^n (\widetilde{\xi_i} f)(x) \widetilde{\xi_i}(x) \in H_x, \qquad f \in C^1(G).
\end{equation}
This too is well-defined independent of the chosen basis.

Finally, let $d$ be the Carnot--Carath\'eodory distance induced by the
sub-Riemannian metric (see \cite[Section 5.2]{blu-book}).
Intuitively, $d(x,y)$ is the length of the shortest horizontal path
joining $x$ and $y$.  The Chow--Rashevskii and ball-box theorems
\cite{montgomery,nagel-stein-wainger} imply that $d(x,y)$ is
finite and that $d$ is a metric which induces the manifold topology on
$G$ (which in turn is just the Euclidean topology on the
finite-dimensional vector space $G = \mathfrak{g}$).  A
straightforward computation shows that $d$ is left-invariant with
respect to the group structure on $G$:
\begin{equation}\label{d-left}
  d(x,y) = d(zx,zy), \qquad x,y,z \in G
\end{equation}
and invariant with respect to the inverse:
\begin{equation}\label{d-inverse}
  d(e, x^{-1}) = d(e,x)
\end{equation}
and also homogeneous with respect to the dilation $\delta_{\lambda}$:
\begin{equation}\label{d-dilate}
  d(\delta_{\lambda}(x), \delta_{\lambda}(y)) = \lambda d(x,y).
\end{equation}
See \cite[Propositions 5.2.4 and 5.2.6]{blu-book} for details.

\subsection{Properties of the heat kernel}\label{subsec:heatkernel}

It is shown in \cite{purplebook} that the Markovian heat semigroup
$e^{s \Delta/4}$ admits a right convolution kernel $\rho_s$,
i.e. $e^{s \Delta/4} f = f \ast \rho_s$; it is also shown that
$\rho_s$ is $C^\infty$ and strictly positive.  This function $\rho_s$
is the \defword{(hypoelliptic) heat kernel} associated to $(G, \langle
\cdot, \cdot \rangle, m)$.  Since $e^{s \Delta/4}$ is Markovian, the
heat kernel measure $\rho_s\,dm$ is a probability measure.  In this
subsection, we collect several properties of the heat kernel from the
literature.

\begin{notation}\label{Lp+}
  For $s > 0$ and $0 < p \le \infty$, we write $L^p(\rho_s)$ as short
  for $L^p(G, \rho_s\, dm)$.  Let
  \begin{align*}
    L^{p+}(\rho_s) &:= \bigcup_{q > p} L^q(\rho_s) \\
    L^{p-}(\rho_s) &:= \bigcap_{q < p} L^q(\rho_s). 
  \end{align*}
  We will say $f_n \to f$ in $L^{p+}(\rho_s)$ (respectively, in
  $L^{p-}(\rho_s)$) if $f_n \to f$ in $L^q(\rho_s)$ for some $q > p$
  (respectively, for all $q < p$).  Also, $W^{1, p+}(\rho_s)$ will
  denote the space of functions $f$ with $f, |\nabla f| \in
  L^{p+}(\rho_s)$.  We shall only have occasion to deal with $C^1$
  functions in $W^{1,p+}$, so we do not discuss weak derivatives here.
\end{notation}

Notice that $L^{p+}(\rho_s)$ and $L^{p-}(\rho_s)$ are vector spaces,
and $L^{\infty-}(\rho_s)$ is an algebra.  By H\"older's inequality, if
$f \in L^{p+}(\rho_s)$ and $g \in L^{\infty -}(\rho_s)$ then $fg \in
L^{p+}(\rho_s)$.  We also note that if $f \in L^{1}(\rho_s)$ is
positive and bounded away from zero then $\log f \in L^{\infty
  -}(\rho_s)$.

\begin{lemma}\label{heat-kernel-inverse}
  The heat kernel $\rho_s$ is invariant under the group inverse
  operation: we have
  $\rho_s(x) = \rho_s(x^{-1})$.
\end{lemma}

\begin{proof}
See  \cite[Theorem III.2.1 (4)]{robinson-book} or the
discussion in \cite[Proposition 3.1
  (3)]{driver-gross-hilbert-spaces-holomorphic}.
\end{proof}

We shall need to make use of sharp upper and lower estimates for the
heat kernel.
\begin{theorem}
\label{heat-kernel-upper-lower}
  For each $0 < \epsilon < 1$ there are constants $C(\epsilon),C'(\epsilon)$ such that
  for every $x \in G$ and $s > 0$,
  \begin{equation}
    \frac{C(\epsilon)}{m(B(e, \sqrt{s}))} e^{-d(e,x)^2/(1-\epsilon)s} \le \rho_s(x) \le \frac{C'(\epsilon)}{m(B(e, \sqrt{s}))} e^{-d(e,x)^2/(1+\epsilon)s}
  \end{equation}
  where $m(B(e,\sqrt{s}))$ is the Lebesgue (Haar) measure of the
  $d$-ball centered at the identity (or any other point of $G$) with
  radius $\sqrt{s}$.
\end{theorem}
\begin{proof}
  The upper bound is Theorem IV.4.2 of \cite{purplebook}.  The lower
  bound is Theorem 1 of \cite{varopoulos-II}.  Note that our choice to
  consider the semigroup $e^{s \Delta/4}$ rather than $e^{s\Delta}$
  accounts for a missing factor of $4$ in the exponents compared to
  the results stated in \cite{purplebook, varopoulos-II}.
\end{proof}

\begin{corollary}\label{polynomial-integrable}
  Any polynomial $p$ in the coordinates $x_{j,k}$ (see Notation
  \ref{adapted-coords}) is in $L^{\infty -}(\rho_s)$.  
\end{corollary}

\begin{proof}
  It suffices to consider $p(x) = x_{j,k}^r$ for some fixed $j,k,r$.
  Let $S = \{x : d(e,x) = 1\}$ be the unit sphere of $d$, and let $C =
  \sup_S |p|$ which is finite by the compactness of $S$.  Then the
  inequality $|p(x)| \le C d(e,x)^{rj}$ holds trivially on $S$.  The
  scaling relations $p(\delta_{\lambda}(x)) = \lambda^{rj} p(x)$ and
  \eqref{d-dilate} now imply that $|p(x)| \le C d(e,x)^{rj}$ for all
  $x \in G$.  The result now follows, via massive overkill, from the
  upper bound in Theorem \ref{heat-kernel-upper-lower}.
\end{proof}

\begin{corollary}\label{nabla-then-E}
  If $|\nabla f| \in L^{p+}(\rho_s)$ then $Ef \in L^{p+}(\rho_s)$.
\end{corollary}

\begin{proof}
  Combine Corollary \ref{E-in-xi-cor} with Corollary \ref{polynomial-integrable}.
\end{proof}

\begin{lemma}\label{heat-kernel-left}
  (Special case of \cite[Theorem IV.3.1]{purplebook}) Let $r \ge 0$
  and $0 < s < t < \infty$.  There is a constant $C$, depending on
  $r,s,t$, such that for all $y \in G$,
  \begin{equation*}
    \sup_{d(e,x) \le r} \rho_s(xy) \le C \rho_t(y).
  \end{equation*}
\end{lemma}

\begin{proof}
  Replacing $x,y$ by $x^{-1}, y^{-1}$ and using Lemma
  \ref{heat-kernel-inverse} and \eqref{d-inverse}, it is enough to
  show the result for $\rho_s(yx)$ instead of $\rho_s(xy)$.
  
  Let $x \in B(e,r)$ be arbitrary.  By the bounds in
  Theorem \ref{heat-kernel-upper-lower}, we have
  \begin{align*}
    \frac{\rho_s(yx)}{\rho_t(y)} \le C''(s,t,\epsilon)
    \exp\left(-\left(\frac{d(e,yx)^2}{(1+\epsilon)s} - \frac{d(e,y)^2}{(1-\epsilon)t}\right)\right).
  \end{align*}
  By the left invariance of the distance $d$ and the triangle
  inequality, we have $d(e,yx) \ge d(e,y) - r$, which yields
  \begin{align*}
    \frac{\rho_s(yx)}{\rho_t(y)} &\le C''(s,t,\epsilon)
    \exp\left( -d(e,y)^2 \left(\frac{1}{(1+\epsilon)s} -
    \frac{1}{(1-\epsilon)t}\right) + \frac{2r d(e,y)}{(1+\epsilon)s} -
    \frac{r^2}{(1+\epsilon)s}\right) \\
    &\le C''(s,t,\epsilon)
    \exp\left( -d(e,y)^2 \left(\frac{1}{(1+\epsilon)s} -
    \frac{1}{(1-\epsilon)t}\right) + \frac{2r
      d(e,y)}{(1+\epsilon)s}\right).
  \end{align*}
  Now since $s < t$, we can take $\epsilon$ sufficiently small that $\frac{1}{(1+\epsilon)s} -
  \frac{1}{(1-\epsilon)t} > 0$.  If we now choose some $r_0$ with
  \begin{equation*}
    r_0 > \left(\frac{1}{(1+\epsilon)s} -
  \frac{1}{(1-\epsilon)t}\right)^{-1} \frac{2r}{(1+\epsilon)s}
  \end{equation*}
  then for all $y$ with $d(e,y) \ge r_0$, the exponent is negative and
  we have $\frac{\rho_s(yx)}{\rho_t(y)} \le C''(s,t,\epsilon)$.  This
  suffices, since by continuity the supremum over all $y \in B(e,r_0)$
  is finite.

\end{proof}

\begin{lemma}\label{hk-ratio-Lp}
  Let $s > 0$, $p > 1$ and $0 < t_0 \le t_1 < \frac{p}{p-1} s$.  Then
    \begin{equation*}
      \sup_{t \in [t_0, t_1]} \frac{\rho_t}{\rho_s} \in L^p(\rho_s).
    \end{equation*}
\end{lemma}

\begin{proof}
For any $t \in [t_0, t_1]$ we have by
Theorem \ref{heat-kernel-upper-lower} that
\begin{align*}
  \left|\frac{\rho_t(x)}{\rho_s(x)}\right|^p \rho_s(x) &= \frac{\rho_t(x)^p}{\rho_s(x)^{p-1}}
   \\ 
  &\le \frac{C'(\epsilon)^p m(B(e, \sqrt{s}))^{p-1}}{C(\epsilon)^{p-1}
     m(B(e, \sqrt{t}))^p}
   \exp\left(-\left(\frac{p}{(1+\epsilon)t} -
  \frac{p-1}{(1-\epsilon)s}\right) d(e,x)^2\right) \\
  &\le \frac{C'(\epsilon)^p m(B(e, \sqrt{s}))^{p-1}}{C(\epsilon)^{p-1}
     m(B(e, \sqrt{t_0}))^p}
   \exp\left(-\left(\frac{p}{(1+\epsilon)t_1} -
  \frac{p-1}{(1-\epsilon)s}\right) d(e,x)^2\right).
\end{align*}
The right side is independent of $t$, and will be integrable on $G$
(with respect to $m$) provided that $\frac{p}{(1+\epsilon)t_1} -
\frac{p-1}{(1-\epsilon)s} > 0$.  But since
\begin{equation*}
  \lim_{\epsilon \to 0} \frac{p}{(1+\epsilon)t_1 } -
  \frac{p-1}{(1-\epsilon)s} = \frac{p}{t_1} - \frac{p-1}{s}  > 0 
\end{equation*}
we can choose $\epsilon$ sufficiently small that this
coefficient is indeed positive.  
\end{proof}

\begin{lemma}
  The heat kernel $\rho_s$ obeys the scaling relation
  \begin{equation}\label{rho-dilate}
    \rho_s(\delta_\lambda(y)) = |\lambda|^{-D} \rho_{s|\lambda|^{-2}}(y).
  \end{equation}
\end{lemma}

\begin{proof}
  This follows from the corresponding scaling properties of the
  semigroup $e^{s\Delta/4}$ \eqref{semigroup-dilate} and of the Haar
  measure $m$ \eqref{m-dilate}.
\end{proof}

\begin{remark}\label{scaling-remark}
  Using \eqref{xi-f-dilation} and \eqref{rho-dilate}, one may verify,
  by replacing $f$ by an appropriate dilation $f \circ \delta_r$, that
  each of the statements \eqref{LSI}, \eqref{sLSI}, \eqref{sHC} in our
  main theorem holds for one $s>0$ iff it holds for all $s>0$, with
  the same constants $c,\beta$.
\end{remark}

\begin{lemma}\label{time-space}
  Suppose $f \in C^2(G) \cap L^1(\rho_s)$ and $Ef, |\nabla f|, \Delta f \in
  L^1(\rho_s)$. Then 
  \begin{equation}
    \int_G E f\, \rho_s\,dm = \frac{s}{2} \int_G \Delta f
    \,\rho_s\,dm.
  \end{equation}
  Moreover, the same result holds if we assume $\Delta f \ge 0$ instead of
  $\Delta f \in L^1(\rho_s)$.
\end{lemma}

\begin{proof}
  Suppose first that $f \in C^\infty_c(G)$.  By
  \eqref{semigroup-dilate} we have
  \begin{equation*}
    \int_G f \circ \delta_{e^{r}}\,\rho_s\,dm = \int_G f \,\rho_{s e^{2r}}\,dm.
  \end{equation*}
  Differentiating under the integral sign at $r = 0$, we
  obtain
  \begin{equation*}
    \int_G Ef\,\rho_s\,dm = 2s \int_G f\,\frac{d}{ds} \rho_s\,dm =
    \frac{s}{2} \int_G f\,\Delta \rho_s\,dm = \frac{s}{2} \int_G
    \Delta f\,\rho_s\,dm.
  \end{equation*}
  Now to show the general case, let $\phi \in C^\infty_c(G)$ be a
  cutoff function with $\phi = 1$ on a neighborhood of the identity
  $e$, and set $\phi_n = \phi \circ \delta_{1/n}$.  Then it is easy to
  check that
  \begin{equation*}
    \nabla \phi_n = \frac{1}{n} (\nabla \phi) \circ \delta_{1/n},
    \qquad \Delta \phi_n = \frac{1}{n^2} (\Delta \phi) \circ
    \delta_{1/n}, \qquad E \phi_n = (E \phi) \circ \delta_{1/n}
  \end{equation*}
  Hence we have, pointwise and boundedly,
  \begin{equation*}
    \phi_n \to 1, \qquad |\nabla \phi_n| \to 0, \qquad \Delta \phi_n
     \to 0, \qquad E \phi_n \to 0,
  \end{equation*}
  the last following from the fact that $E \phi = 0$ on a neighborhood
  of $e$.  Applying our result to $\phi_n f$, we have
  \begin{align*}
    &\int_G E \phi_n \cdot f\,\rho_s\,dm +\int_G \phi_n \cdot Ef
    \,\rho_s\,dm \\
    &\quad = \int_G \Delta \phi_n \cdot f \,\rho_s\,dm
    + \int_G g(\nabla \phi_n, \nabla f)\,\rho_s\,dm + \int_G \phi_n \cdot
    \Delta f\,\rho_s\,dm.
  \end{align*}
  By dominated convergence, using the integrability assumptions on $f$
  and its derivatives, letting $n \to \infty$ gives the desired
  identity.

  If we only assume $\Delta f \ge 0$, then if we choose $\phi_n$ with
  a little more care, we can get $\phi_n \to 1$ monotonically.  Then
  we can repeat the argument above, in which we have $\int_g \phi_n
  \cdot \Delta f\,\rho_s\,dm \to \int_G \Delta f\,\rho_s\,dm$ by
  monotone convergence instead of dominated convergence.
\end{proof}

\section{Log-subharmonic functions}\label{sec:LSH}

In general, a function $f : G \to [0,\infty)$ on $G$ is said to be
  log-subharmonic (LSH) if $\log f$ is subharmonic with respect to the
  sub-Laplacian $\Delta$.  There are many possible notions of
  subharmonicity in this setting.  In this paper, we shall work primarily
  with a strong ``classical'' notion of subharmonicity, in order to
  avoid obscuring the main ideas with technicalities; but see Section
  \ref{weak-lsh} below, where we discuss how the results of this paper can be
  applied to functions which are log-subharmonic in a weaker sense.

\begin{definition}\label{lsh-def}
  Suppose $f \in C^2(G)$.  We will say $f$ is \defword{subharmonic}
  if $\Delta f \ge 0$.  We will say $f$ is \defword{log-subharmonic
    (LSH)} if $f > 0$ and $\Delta \log f \ge 0$.
\end{definition}

\begin{lemma}\label{LSH-ineq-lemma}
  If $f \in C^2(G)$ and $f > 0$, then $f$ is LSH if and only if
  \begin{equation}\label{LSH-ineq}
    \Delta f \ge \frac{|\nabla f|^2}{f}.
  \end{equation}
  In particular, LSH functions are subharmonic.
\end{lemma}

\begin{proof}

  By the chain and product rules,
  \begin{equation*}
    \Delta \log(f) = -\frac{|\nabla f|^2}{f^2} + \frac{\Delta f}{f} =
    \frac{1}{f} \left( - \frac{|\nabla f|^2}{f} + \Delta f\right)
  \end{equation*}
  so that $\Delta \log(f) \ge 0$ iff $\Delta f \ge \frac{|\nabla f|^2}{f}$.
\end{proof}

\begin{proposition}\label{lsh-properties}
  Suppose $f,g$ are LSH.  The following functions are LSH:
  \begin{enumerate}
  \item \label{lsh-const} \label{obvious-start} Positive constants
  \item \label{lsh-prod} $fg$
  \item \label{lsh-power} $f^p$ for any $p > 0$ \label{obvious-end}
  \item \label{lsh-sum} $f+g$
  \item \label{lsh-dilate} $f \circ \delta_{\lambda}$ for any $\lambda > 0$
  \end{enumerate}
\end{proposition}

\begin{proof}
  Items \ref{obvious-start}--\ref{obvious-end} are immediate.  

  For item \ref{lsh-sum}, we use a trick suggested in
  \cite[Proposition 2.2]{gkl2010}.  We have that $u = \log f$ and $v =
  \log g$ are subharmonic.  Fix $x \in G$ and assume without loss of
  generality that $\Delta u(x) \ge \Delta v(x)$.  Now
  \begin{equation*}
    \Delta \log(f+g) = \Delta \log(e^u + e^v) = \Delta [v +
    \log(e^{u-v}+1)] \ge \Delta \log(e^{u-v}+1).
  \end{equation*}
  Let $\psi(t) = \log(e^t+1)$ and note that $\psi'(t) =
  \frac{e^t}{1+e^t} > 0$ and $\psi''(t) = \frac{e^t}{(1+e^t)^2} > 0$.
  By the chain and product rules, we have
  \begin{equation*}
    \Delta \log(e^{u-v}+1) = \Delta \psi(u-v) = \psi''(u-v) |\nabla
    [u-v]|^2 + \psi'(u-v) \Delta[u-v].
  \end{equation*}
  This is nonnegative at $x$ since by assumption $\Delta u(x) \ge
  \Delta v(x)$.

  Item \ref{lsh-dilate} is an immediate consequence of
  \eqref{Delta-dilate} which implies that
  \begin{equation*}
    \Delta \log(f \circ
    \delta_\lambda) = \lambda^2 (\Delta \log(f)) \circ \delta_{\lambda}.
    \end{equation*}
\end{proof}

\begin{lemma}\label{convolution-LSH}
  If $f$ is LSH and $\varphi \in C^\infty_c(G)$ is nonnegative then
  $\varphi \ast f$ is LSH.
\end{lemma}

\begin{proof}
  We will show that $\varphi \ast f$ satisfies \eqref{LSH-ineq}.  By
  rescaling, let us suppose without loss of generality that $\int_G
  \varphi \,dm = 1$.  Fix an orthonormal basis $\xi_1, \dots, \xi_n$
  for $V_1$.  From \eqref{Delta-def} and \eqref{convolution-diff}, we
  have $\Delta [ \varphi \ast f] = \varphi \ast \Delta f$, and so
  Lemma \ref{LSH-ineq-lemma} gives
  \begin{equation*}
    \Delta [ \varphi \ast f] = \varphi \ast \Delta f \ge \varphi \ast
    \frac{|\nabla f|^2}{f} = \sum_{i=1}^n \varphi \ast \frac{(\widetilde{\xi_i} f)^2}{f}. 
  \end{equation*}
  Applying the multivariate Jensen inequality with the convex function
  $\psi(u,v) = u^2/v$ and the probability measure $\varphi\,dm$, we have
  \begin{align*}
    \left(\varphi \ast \frac{(\widetilde{\xi_i} f)^2}{f}\right)(x) &= \int_G
    \varphi(y) \frac{(\widetilde{\xi_i} f(y^{-1} x))^2}{f(y^{-1} x)}\,dy \\ &\ge
    \frac{\left(\int_G \varphi(y) \widetilde{\xi_i} f(y^{-1} x)\,dy\right)^2}{\int_G
      \varphi(y) f(y^{-1} x)\,dy} \\
    &= \frac{(\varphi \ast (\widetilde{\xi_i} f))(x)^2}{(\varphi \ast f)(x)} \\
    &= \frac{(\widetilde{\xi_i}[\varphi \ast f](x))^2}{(\varphi \ast f)(x)}
  \end{align*}
  using \eqref{convolution-diff} again, since $\widetilde{\xi_i}$ is left-invariant.
  Thus we have
  \begin{equation*}
    \Delta [ \varphi \ast f] \ge \sum_{i=1}^n \frac{(\widetilde{\xi_i}[\varphi \ast
        f])^2}{\varphi \ast f} = \frac{|\nabla (\varphi \ast
      f)|^2}{\varphi \ast f} 
  \end{equation*}
  and so by Lemma \ref{LSH-ineq-lemma}, $\varphi \ast f$ is LSH.
\end{proof}

\section{Examples and special cases}\label{examples-sec}

\subsection{Euclidean space}

\begin{example}\label{euclidean-ex}
  As a trivial example, $G = \mathbb{R}^n$ with Euclidean addition is
  an (abelian) stratified Lie group of step 1.  (Indeed, these are all
  the step 1 stratified Lie groups.)  Here the Lie bracket is zero and
  the dilation is $\delta_\lambda(x) = \lambda x$.  If we equip $V_1 =
  \mathfrak{g} = \mathbb{R}^n$ with the Euclidean inner product, then
  the sub-Laplacian and sub-gradient are the usual Euclidean Laplacian
  and gradient, and the Carnot--Carath\'eodory distance $d$ is
  Euclidean distance.  The heat kernel $\rho_s$ is the Gaussian
  density, appropriately scaled.  (Note that, in our normalization,
  standard Gaussian density corresponds to $s=2$.)

  As such, the results of this paper include statements about Gaussian
  measure on Euclidean space, similar to those obtained in
  \cite{gkl2010, gkl2015}.  It is well known that \eqref{LSI} is true
  for Gaussian measures \cite{gross75}, with constant $c =
  \frac{1}{2}$.  
\end{example}

\subsection{The Heisenberg group}

\begin{example}\label{heisenberg-ex}
  The simplest nontrivial example of a stratified Lie group is the
  3-dimensional real \defword{Heisenberg group} $G = \mathbb{H}^3$, which we may
  realize as $\mathbb{R}^3$ equipped with the group operation
  \begin{equation}\label{heisenberg-group-op}
    (x_1, x_2, x_3)(x_1', x_2', x_3') = \left(x_1 + x_1', x_2 + x_2',
    x_3 + x_3' + \frac{1}{2}(x_1 x_2' - x_2 x_1')\right).
  \end{equation}
  If we let $\xi_i = \pp{x_i} \in \mathfrak{g} = T_e G$ for $i=1,2,3$,
  the Lie bracket is given by
  \begin{equation*}
    [\xi_1, \xi_2] = \xi_3, \qquad [\xi_1, \xi_3] = [\xi_2, \xi_3] = 0
  \end{equation*}
  so we have the decomposition $\mathfrak{g} = V_1 \oplus V_2$ where
  $V_1 = \spanop\{\xi_1, \xi_2\}$, $V_2 = \spanop\{\xi_3\}$.  Thus the
  Heisenberg group is stratified of step 2.  A natural inner product
  $\langle \cdot, \cdot \rangle$ on $V_1$ is given by taking $\xi_1,
  \xi_2$ to be orthonormal.

  The corresponding
  left-invariant vector fields are given by
  \begin{equation*}
    \widetilde{\xi_1} = \pp{x_1} - \frac{1}{2} x_2 \pp{x_3}, \quad 
    \widetilde{\xi_2} = \pp{x_2} + \frac{1}{2} x_1 \pp{x_3}, \quad 
    \widetilde{\xi_3} = \pp{x_3}.
  \end{equation*}

  The Heisenberg group $\mathbb{H}^3$ was the first nontrivial
  stratified Lie group that was shown to satisfy the logarithmic
  Sobolev inequality \eqref{LSI}.  This statement can be found in
  \cite{bbbc-jfa} and follows from heat semigroup
  gradient bounds previously established in \cite{li-jfa}, via a
  variant of a standard
  $\Gamma_2$-calculus argument from
  \cite{bakry-emery-short} or \cite[pp.~69--70]{bakry-sobolev}.  A key
  ingredient is sharp upper and lower heat kernel estimates, obtained
  in \cite{li-heatkernel}.  As such, our Theorem \ref{main-combined}
  implies that 
  \eqref{sLSI} and \eqref{sHC} are satisfied by $\mathbb{H}^3$ as
  well.

  The Heisenberg group construction immediately generalizes to the
  family of \defword{Heisenberg--Weyl groups} $\mathbb{H}^{2n+1}$, which is
  realized as $\mathbb{R}^{2n+1}$ with a group operation defined again
  by \eqref{heisenberg-group-op}, where now we take $x_1, x_2 \in
  \mathbb{R}^n$.

  The Heisenberg and Heisenberg--Weyl groups are examples of H-type
  groups, which we discuss next.
\end{example}

\subsection{H-type groups}

\begin{example}\label{h-type-example}
  Suppose that $G$ is a (real) stratified Lie group of step $2$.  Let
  the inner product $\langle \cdot, \cdot \rangle$ on $V_1$ be
  extended to an inner product on all of $\mathfrak{g} = V_1 \oplus
  V_2$, still denoted by $\langle \cdot, \cdot \rangle$, for which the
  decomposition $\mathfrak{g} = V_1 \oplus V_2$ is orthogonal.  For
  each $z \in V_2$, define the linear map $J_z : V_1 \to V_1$ by
  $\langle J_z v, w \rangle = \langle z, [v,w] \rangle$.  We say that
  $(G, \langle \cdot, \cdot \rangle)$ is \defword{H-type} if, for each
  $z \in V_2$ with $\langle z,z \rangle = 1$, the map $J_z$ is a
  partial isometry.  H-type groups were introduced in \cite{kaplan80};
  see \cite[Chapter 18]{blu-book} for more background on these groups.
  The Heisenberg and Heisenberg--Weyl groups discussed in Example
  \ref{heisenberg-ex} are H-type (indeed, the H stands for
  Heisenberg).
\end{example}

  H-type groups satisfy the same type of heat semigroup gradient
  bounds as the Heisenberg group $\mathbb{H}^3$.  This was shown
  independently in \cite{eldredge-gradient,li-gradient-h-type}; for
  the required heat kernel estimates, see
  \cite{eldredge-precise-estimates,li-heat-h-type}.  Thus, such groups
  satisfy \eqref{LSI} as well, by the same general argument given in
  \cite{bbbc-jfa}.  As we noted in Corollary \ref{H-type-cor}, our
  Theorem \ref{main-combined} then implies that \eqref{sLSI} and
  \eqref{sHC} are also true in H-type groups.  We do not know of any
  further examples of stratified Lie groups where \eqref{LSI} has been
  proved.
  
  \subsection{Complex stratified Lie groups}

\begin{example}\label{complex-example}
  Suppose that $G$ is a stratified Lie group which is also a complex
  Lie group, so that the Lie algebra $\mathfrak{g}$ admits a complex
  structure $J : \mathfrak{g} \to \mathfrak{g}$ satisfying $[Jv, w] =
  J[v,w]$.  Then $\mathfrak{g}$ is a complex vector space and the
  subspaces $V_i$ in the decomposition \eqref{decomp} are complex
  vector spaces as well.  The complex structure on $\mathfrak{g}$
  induces a complex manifold structure on $G$ for which the
  exponential map is holomorphic.  In this setting, it is natural to
  ask that the inner product $\langle \cdot, \cdot \rangle$ on $V_1$
  be compatible with the complex structure, by being Hermitian:
  $\langle Jv, w \rangle = -\langle v, Jw \rangle$.  We call such $G$
  a \defword{complex stratified Lie group}.  A simple example is the
  complex Heisenberg group $\mathbb{H}^3_{\mathbb{C}}$, or the complex
  Heisenberg--Weyl groups $\mathbb{H}^{2n+1}_{\mathbb{C}}$.
\end{example}

\begin{lemma}\label{holomorphic-lsh}
  Suppose $G$ is a complex stratified Lie group.  Let $f : G \to
  \mathbb{C}$ be holomorphic.  Then for any $\epsilon > 0$, the
  function $g = \sqrt{|f|^2+\epsilon}$ is LSH.
\end{lemma}

We cannot say that $|f|$ itself is LSH by our definition, because
$|f|$ need not be either $C^2$ nor strictly positive, but it is weakly
LSH in the sense of Section \ref{weak-lsh}; see Proposition
\ref{holomorphic-wlsh}.

\begin{proof}
  Let $x \in G$ and suppose for the moment that $f(x) \ne 0$.  By
  continuity, there is a disk $D \subset \mathbb{C} \setminus \{0\}$
  and an open neighborhood $U$ of $x$ such that $f(U) \subset D$.  Let
  $L(z)$ be a branch of the complex logarithm which is holomorphic on
  $D$.  Then $L \circ f$ is holomorphic on $U$.  We are assuming that
  the inner product on $V_1$ is Hermitian, so the real and imaginary
  parts of any holomorphic function are harmonic with respect to the
  sub-Laplacian $\Delta$ (see \cite{eldredge-gross-saloff-coste-dila}
  for further details).  Thus $\Delta \log |f| = \Delta \Re L \circ
  f = 0$ on $U$.  It follows that $\Delta \log g \ge 0$ on $U$ (see
  Proposition \ref{lsh-properties} items \ref{lsh-power} and
  \ref{lsh-sum}).  So we have shown $\Delta \log g \ge 0$ on
  $\{f \ne 0\}$.
  But since $f$ is holomorphic, $\{f \ne 0\}$ is dense in $G$
  (unless $f \equiv 0$ in which case the statement is trivial).  Since
  $g$ is strictly positive and $C^\infty$, $\Delta \log g$ is
  continuous.  Thus we have $\Delta \log g \ge 0$ everywhere.
\end{proof}

\begin{corollary}
  Let $G$ be a complex stratified Lie group satisfying \eqref{sHC}.
  Then \eqref{sHC} also holds for all holomorphic $f \in L^q(\rho_s)$.
\end{corollary}

\begin{proof}
  Apply \eqref{sHC} to $\sqrt{|f|^2+\epsilon}$, and let $\epsilon
  \downarrow 0$ using Lemma \ref{holomorphic-lsh} and dominated convergence.
\end{proof}

In particular, by Theorem \ref{main-combined}, this holds whenever $G$
satisfies \eqref{LSI}.  This implication was one of the main results
of \cite{eldredge-gross-saloff-coste-dila}, which also gave a density
argument for holomorphic $L^p$ that can be used to show that
\eqref{sHC} also holds for holomorphic $f \in L^p(\rho_s)$.  See the
related discussion in Remark \ref{why-Lq}.  Unfortunately, we do not
know of any similar density results for LSH functions in the real
case.

For a complex stratified Lie group $G$, the vector field $E$ has an
additional significance: as shown in
\cite{eldredge-gross-saloff-coste-dila}, it is the holomorphic
projection of the Ornstein--Uhlenbeck operator $A$.  In the special
case of $\mathbb{C}^n$ with the Gaussian heat kernel, if $f$ is
holomorphic then we actually have $Af = Ef$ because the Laplacian term
vanishes.  For more general complex stratified groups, this is no
longer true because $Af$ may fail to be holomorphic, but its
$L^2(\rho_s)$ projection onto the holomorphic functions equals $Ef$.

  There is not much overlap between the complex and H-type Lie groups;
  we showed in \cite{eldredge-complex-h-type} that the only complex
  Lie groups which are also H-type are the complex Heisenberg--Weyl
  groups $\mathbb{H}^{2n+1}_{\mathbb{C}}$.  As such, these are the
  only complex stratified Lie groups for which \eqref{LSI} is
  currently known to hold.

\section{Convolution and approximation}\label{sec-convolution}

\begin{lemma}\label{convolution-Lp+}
  If $f \in L^{p+}(\rho_s)$, $p \ge 1$, and $\varphi \in C_c(G)$ then
  $\varphi \ast f \in L^{p+}(\rho_s)$.
\end{lemma}

\begin{proof}
  By considering positive and negative parts, we can assume without
  loss of generality that $\varphi \ge 0$.  Also, by rescaling we can
  assume $\int_G \varphi\,dm = 1$.

  Let $q > p$ and $r > 1$ be so small that $f \in L^{qr}(\rho_s)$.
  Let $\frac{1}{r} + \frac{1}{r^*} = 1$ and choose any $t$ with $s < t
  < rs = \frac{r^*}{r^*-1} s$. Then by Lemma \ref{hk-ratio-Lp} we have
  $\frac{\rho_t}{\rho_s} \in L^{r^*}(\rho_s)$.  Next, let $K$ be the
  support of $\varphi$, which is compact, and use Lemma
  \ref{heat-kernel-left} to choose $C$ such that $\sup_{z \in K}
  \rho_s(zy) \le C \rho_t(y)$ for all $y \in G$.

  To start, use Jensen's inequality with the probability measure
  $\varphi\,dm$ to see that
  \begin{equation*}
    |(\varphi \ast f)(x)|^q = \left|\int_G \varphi(z) f(z^{-1}
    x)\,dz\right|^q \le \int_G \varphi(z) |f(z^{-1} x)|^q\,dz
  \end{equation*}
  so that, by Fubini's theorem,
  \begin{align*}
    \|\varphi \ast f\|_{L^q(\rho_s)}^q &\le \int_K \int_G \varphi(z)
    |f(z^{-1} x)|^q \rho_s(x)\,dx\,dz \\
    &= \int_K \int_G \varphi(z) |f(y)|^q \rho_s(zy)\,dy\,dz
  \end{align*}
  making the change of variables $y = z^{-1} x$ and using the
  translation invariance of $m$.  Now for all $z$ in the support $K$
  of $\varphi$, we have $\rho_s(zy) \le C \rho_t(y)$.  Since $\int_G
  \varphi(z)\,dz = 1$ by assumption, we now have
  \begin{align*}
    \|\varphi \ast f\|_{L^q(\rho_s)}^q &\le C \int_G |f(y)|^q
    \rho_t(y)\,dy \\
    &= C\int_G |f(y)|^q \frac{\rho_t(y)}{\rho_s(y)} \rho_s(y)\,dy \\
    &\le C \|f\|_{L^{qr}(\rho_s)}^q \left\|\frac{\rho_t}{\rho_s}\right\|_{L^{r^*}(\rho_s)}
  \end{align*}
  by H\"older's inequality.  By our choices of $t,q,r$, both norms are finite.  
\end{proof}

\begin{lemma}\label{convolution-E}\label{convolution-deriv}
  Suppose $f \in L^{p+}(\rho_s)$, $p \ge 1$, and $\varphi \in C^\infty_c(G)$.
  Then for any $\xi \in \mathfrak{g}$, we have
  $\widetilde{\xi}[\varphi \ast f] \in L^{p+}(\rho_s)$.  As a
  consequence, we also have
  \begin{equation*}
    \left|\nabla[\varphi \ast f]\right|, \, \Delta[\varphi \ast f], \, E[\varphi \ast f] \in L^{p+}(\rho_s).
  \end{equation*}
\end{lemma}

\begin{proof}
  By writing $\xi$ as a linear combination of an adapted basis
  $\{\xi_{j,k}\}$ and using Lemma \ref{poly-vector-fields}, we can
  write $\widetilde{\xi} = \sum_{j,k} c_{j,k} \widehat{\xi_{j,k}}$ for
  some polynomials $c_{j,k}$.  In particular, by Corollary
  \ref{polynomial-integrable} we have $c_{j,k} \in L^{\infty
    -}(\rho_s)$.  Now using \eqref{convolution-diff}, we
  have
  \begin{equation*}
    \widetilde{\xi}[\varphi \ast f] = \sum_{j,k} c_{j,k}
    \widehat{\xi_{j,k}}[\varphi \ast f] = \sum_{j,k} c_{j,k}
    \cdot [(\widehat{\xi_{j,k}} \varphi) \ast f]
  \end{equation*}
  which is in $L^{p+}(\rho_s)$ by Lemma \ref{convolution-Lp+}.
  
  The desired statement for $\Delta[\varphi \ast f]$ follows by
  applying this twice to get $\widetilde{\xi}^2 [\varphi \ast f] \in
  L^{p+}(\rho_s)$, then summing over an orthonormal basis $\{\xi_i\}$
  for $V_1$.  For $E[\varphi \ast f]$, use Corollary \ref{nabla-then-E}.
\end{proof}

\begin{lemma}\label{convolution-converge}
  Suppose $f \in L^{1+}(\rho_s)$.  There is a sequence of nonnegative
  $\varphi_n \in C^\infty_c(G)$ such that $\varphi_n * f \to f$ almost
  everywhere and in $L^{1+}(\rho_s)$. If moreover $f \in C(G)$ the
  convergence is uniform on compact sets.
\end{lemma}

\begin{proof}
  Let $\varphi \in C^\infty_c(G)$ be nonnegative with $\int_G
  \varphi\,dm = 1$, and let $\varphi_n = n^D \varphi \circ \delta_n$,
  so that $\varphi_n$ is a sequence of standard mollifiers. It is
  standard that $\varphi_n \ast f \to f$ almost everywhere, after
  passing to a subsequence if necessary, and that the convergence is
  uniform on compact sets if $f$ is continuous.

  Now we note that the $\varphi_n$ are all supported in some compact
  neighborhood $K$ of the identity.  As in the proof of Lemma
  \ref{convolution-Lp+}, if we choose $q, r > 1$ such that $f \in
  L^{qr}(\rho_s)$, then we can choose $C,t$, independent of $n$, such
  that
  \begin{equation*}
    \|\varphi_n \ast f\|_{L^q(\rho_s)}^q \le C \|f\|_{L^{qr}(\rho_s)}^q
    \left\|\frac{\rho_t}{\rho_s}\right\|_{L^{r^*}(\rho_s)}
  \end{equation*}
  In particular, if $1 < q' < q$ then $\{|\varphi_n \ast f|^{q'} : n \ge 1\}$ is uniformly
  integrable with respect to $\rho_s\,dm$, hence $\varphi_n \ast f$ converges in
  $L^{q'}(\rho_s)$.
\end{proof}

\section{Differentiation under the integral sign}

As mentioned in Section \ref{sec-intro}, the strong logarithmic
Sobolev inequality \eqref{sLSI} is essentially an infinitesimal
version of the strong hypercontractivity inequality \eqref{sHC}.
Thus, at a purely formal level, the equivalence between them is
completely natural, and consists mainly of differentiating under the
integral sign.  The difficulty is to verify that this is justified.

The following abstract lemma is a general principle for
differentiating under the integral sign.  We have not seen this
particular form in the literature, so we give the proof.

\begin{lemma}\label{dui}
  Let $(X,\mu)$ be a probability space, and let $F : [0,T] \times X
  \to \mathbb{R}$ be jointly measurable.  Suppose that for each $x \in
  X$, we have $F(\cdot, x) \in C^1([0,T])$, so that $\partial_t F :
  [0,T] \to X$ is also jointly measurable.  Furthermore, suppose that
  the family of functions $\left\{ \partial_t F(t, \cdot) : 0 \le t \le T
  \right\}$ is uniformly integrable on $(X,\mu)$.  Set
  \begin{equation}
    \begin{split}
      v(t) &= \int_X F(t,x) \,\mu(dx) \\
      w(t) &= \int_X \partial_t F(t,x)\,\mu(dx).
    \end{split}
  \end{equation}
  Then $v \in C^1([0,T])$ and $v'(t) = w(t)$ on $[0,T]$.
\end{lemma}

We use the term ``uniformly integrable'' here in the probabilist's
sense: a family of functions $\mathcal{G}$ is uniformly integrable with
respect to $\mu$ iff
\begin{equation*}\lim_{M \to \infty} \sup_{g \in \mathcal{G}}
  \int_{|g| \ge M} |g| \,d\mu = 0.
\end{equation*}
This is the necessary and sufficient hypothesis for the Vitali
convergence theorem.  In particular, a uniformly integrable family is
bounded in $L^1(\mu)$.  We also recall, for later use, the fact that
if $\sup_{g \in \mathcal{G}} \|g\|_{L^p(\mu)} < \infty$ for some $p>1$
then $\mathcal{G}$ is uniformly integrable.

\begin{proof}[Proof of Lemma \ref{dui}]
  First, by the Vitali convergence theorem, $w$ is continuous on
  $[0,T]$.

  The uniform integrability also implies that $\partial_t F(t, \cdot)$
  is uniformly $L^1$ bounded, so we have $\| \partial_t F(t,
  \cdot)\|_{L^1(\mu)} \le K$ for some finite $K$.  Thus for
  any $0 \le \tau \le T$ we have
  \begin{equation*}
    \int_0^\tau \int_X |\partial_t F(t, x)|\,\mu(dx)\,dt \le KT < \infty. 
  \end{equation*}
  So by Fubini's theorem and the first fundamental theorem of calculus, we have
  \begin{align*}
    \int_0^\tau w(t)\,dt &= \int_0^\tau \int_X \partial_t
    F(t,x)\,\mu(dx)\,dt \\
    &= \int_X \int_0^\tau \partial_t F(t,x)\,dt\,\mu(dx) \\
    &= \int_X (F(\tau,x) - F(0,x))\,\mu(dx) \\
    &= v(\tau) - v(0).
  \end{align*}
  Hence by the second fundamental theorem of calculus, $v$
  is differentiable and $v' = w$.  
\end{proof}

The following lemma, which follows from the heat kernel bounds in
Section \ref{subsec:heatkernel}, will be
convenient in verifying the uniform integrability hypothesis for our
applications.

\begin{lemma}\label{dilate-ui}
  Suppose $f \in L^{p+}(\rho_s)$.  Then for some $q > p$ and any $T <
  \infty$ we have
  \begin{equation*}
    \sup_{0 \le t \le T} \|e^{-tE} f\|_{L^q(\rho_s)} < \infty.
  \end{equation*}
  In particular, if $f \in L^{1+}(\rho_s)$, then $\{e^{-tE} f : 0 \le t \le
  T\}$ is uniformly integrable with respect to $\rho_s \,dm$.
\end{lemma}

\begin{proof}
  Choose $q > p$ and $r > 1$ so small that $f \in L^{qr}(\rho_s)$.
  Let $r^* = \frac{r}{r-1}$.  By Lemma
  \ref{hk-ratio-Lp}, taking $t_0 = s e^{-2T}$ and $t_1 = s$, we have $\sup_{0 \le t \le T} \rho_{se^{-2t}}/\rho_s \in L^{r^*}(\rho_s)$.
  Then for any $t \in [0,T]$ we have
  \begin{align*}
    \|e^{-tE} f\|_{L^q(\rho_s)}^q &= \int_G |f \circ
    \delta_{e^{-t}}|^q \rho_s\,dm \\
    &= \int_G |f|^q \rho_{s e^{-2t}}\,dm && \text{by
      \eqref{rho-dilate} and \eqref{m-dilate-integrate}}\\
    &\le \int_G |f|^q \left(\sup_{0 \le t \le T}
    \frac{\rho_{s e^{-2t}}}{\rho_s}\right) \rho_s\,dm \\
    &\le \|f\|_{L^{qr}(\rho_s)}^q \left\|\sup_{0 \le t \le T}
    \frac{\rho_{s e^{-2t}}}{\rho_s}\right\|_{L^{r^*}(\rho_s)}
  \end{align*}
  which is independent of $t$ and finite.
\end{proof}

\begin{lemma}\label{ui-lemma}
  Suppose $r \in C^1([0,T])$ with $1 \le r(t) \le q$, and suppose $f
  \in W^{1, q+}(\rho_s)$ is positive.  Set $f_t = e^{-tE} f^{r(t)}$.
  Then the functions
  \begin{equation*}
    f_t, \quad f_t \log f_t, \quad |\nabla f_t|, \quad E f_t, \qquad 0
    \le t \le T
  \end{equation*}
  are all uniformly bounded in $L^p(\rho_s)$ norm for some $p > 1$.
  In particular, they are uniformly integrable.
\end{lemma}

We note that the conclusion of this lemma implies $f_t \in
W^{1,1+}(\rho_s)$ for each $0 \le t \le T$.

\begin{proof}
  For $f_t$, note that $1+f^q$ is in $L^{1+}$, so by Lemma
  \ref{dilate-ui} we have that the family $\{e^{-tE}[1+ f^q] : 0 \le t
  \le T\}$ is uniformly bounded in $L^p$ for some $p > 1$.  But
  $f^{r(t)} \le 1 + f^q$ for each $t$, so $f_t \le e^{-tE}[1+f^q]$,
  and $f_t$ is uniformly bounded in $L^p$.

  For $f_t \log f_t$, note that since $f^q \in L^{1+}$, we also have
  $(1+f^q) \log f^q \in L^{1+}$.  Then, since
  \begin{equation*}
    |\log f^{r(t)}| = r(t) |\log f| \le q |\log f| = |\log f^q|
  \end{equation*}
  we have $|f^{r(t)} \log f^{r(t)}| \le |(1+f^q) \log f^q|$.  By the
  same argument as in the previous case, $f_t \log f_t$ is uniformly
  bounded in some $L^p$.

  For $|\nabla f_t|$, note that
  \begin{equation*}
    |\nabla f_t| = r(t) e^{-tE} [f^{r(t)-1}] \left|\nabla [e^{-tE} f]\right| =
    r(t) e^{-t} e^{-tE}\left[f^{r(t)-1} |\nabla f|\right]   
  \end{equation*}
  using \eqref{xi-f-dilation}.  Now $f^{r(t)-1} \le 1+f^{q-1}$, and we
  have $(1+f^{q-1})|\nabla f| \in L^{1+}$ by H\"older's inequality.
  So by Lemma \ref{dilate-ui}, $\left\{e^{-tE} \left[(1+f^{q-1})|\nabla f|\right]\right\}$
  is uniformly bounded in some $L^p$, $p>1$, and the same thus holds for
  $|\nabla f_t|$.

  Since $|\nabla f| \in L^{q+}$, we have $Ef \in L^{q+}$ as well, by
  Lemma \ref{nabla-then-E}.  So a similar argument applies for $E f_t$
  as for $|\nabla f_t|$, noting that
  \begin{equation*}
    E f_t = r(t) e^{-tE} [f^{r(t)-1}] Ee^{-tE} f = r(t) e^{-tE}
    \left[f^{r(t)-1} Ef\right].
  \end{equation*}

\end{proof}

\begin{lemma}\label{key-diff}
  Again suppose $r \in C^1([0,T])$ with $1 \le r(t) \le q$, and suppose $f
  \in C^1(G) \cap W^{1, q+}(\rho_s)$ is positive.  Set $f_t = e^{-tE} f^{r(t)}$.  Let
  \begin{equation}
    \begin{split}\label{key-vw-def}
      v(t) &= \int_G f_t\,\rho_s\,dm = \|e^{-tE}
      f\|_{L^{r(t)}(\rho_s)} \\
      w(t) &= \int_G \partial_t f_t\,\rho_s\,dm = \int_G \left[ -E f_t +
      \frac{r'(t)}{r(t)} f_t \log f_t\right]\,\rho_s\,dm. 
    \end{split}
  \end{equation}
  Then $v \in C^1([0,T])$ and $v'(t) = w(t)$ on $[0,T]$.
\end{lemma}

\begin{proof}
  Differentiate under the integral sign using Lemma \ref{dui}, with
  $F(t,x) = f_t(x)$.  The continuity of $\partial_t f_t$ follows from
  the assumption that $f \in C^1(G)$, and the uniform integrability
  hypothesis is verified by Lemma \ref{ui-lemma}.
\end{proof}

\section{Proofs of the main results}

\subsection{LSI implies sLSI}

\begin{theorem}
  \label{LSI-implies-sLSI}
  In any stratified Lie group $G$, if \eqref{LSI} holds, then
  \eqref{sLSI} holds, with the same constants $c,\beta$.
\end{theorem}

\begin{proof}
  Suppose $f \in LSH \cap W^{1,1+}(\rho_s)$; by Corollary
  \ref{nabla-then-E} we have $Ef \in L^{1+}(\rho_s)$.  Since
  LSH functions are subharmonic ($\Delta f \ge 0$), we can apply
  Lemmas \ref{LSH-ineq-lemma} and \ref{time-space} to obtain
  \begin{equation*}
    \int_G \frac{|\nabla f|^2}{f}\,\rho_s\,dm \le \int_G \Delta
    f\,\rho_s\,dm = \frac{2}{s} \int_G Ef\,\rho_s\,dm.
  \end{equation*}
  Inserting this inequality into \eqref{LSI} yields \eqref{sLSI}.
\end{proof}

\subsection{sHC implies sLSI}

\begin{theorem}\label{sHC-implies-sLSI}  In any stratified Lie group
  $G$, if \eqref{sHC} holds, then \eqref{sLSI} holds, with the same
  constants $c,\beta$.
  
\end{theorem}

\begin{proof}
  Suppose \eqref{sHC} holds with constants $c,\beta$.  Fix $f \in LSH
  \cap W^{1,1+}(\rho_s)$.  Set
  $r(t) = e^{t/c}$ and choose $T > 0$ so small that $f \in W^{1, r(T)+}(\rho_s)$.

  For $t \in [0,T]$, applying \eqref{sHC} with $p=1$ and $q = r(T)$
  yields
  \begin{equation}\label{sHC-in-proof}
    \|e^{-tE} f\|_{L^{r(t)}(\rho_s)} \le M(t) \|f\|_{L^1(\rho_s)}
  \end{equation}
  where
  \begin{equation}
    M(t) := M(1, r(t)) = \exp(\beta \cdot (1 - e^{-t/c})).
  \end{equation}

  Define $v(t), w(t)$ as in \eqref{key-vw-def}, and set
  \begin{equation}\label{alpha-def}
    \alpha(t) = \frac{1}{M(t)} \|e^{-tE} f\|_{L^{r(t)}(\rho_s)} =
    \frac{1}{M(t)} v(t)^{1/r(t)}.
  \end{equation}
  Note that $\alpha(0) = \|f\|_{L^1(\rho_s)}$, so \eqref{sHC-in-proof}
  says that $\alpha(t) \le \alpha(0)$ for all $t \in [0,T]$.

  Now applying Lemma \ref{key-diff} with $q=r(T)$, we have that $v$ is
  continuously differentiable on $[0,T]$; hence so is $\alpha(t)$.  As
  such, we must have
  \begin{equation}\label{alpha-prime}
    \begin{split}
    0 \ge \alpha'(0) &= \frac{-M'(0)}{M(0)^2} v(0)^{1/r(0)} \\
    &\quad + \frac{1}{M(0)} \frac{1}{r(0)} v(0)^{(1/r(0)) - 1} v'(0)
    \\
    &\quad + \frac{1}{M(0)} v(0)^{1/r(0)} \log v(0) \frac{-r'(0)}{r(0)^2}.
    \end{split}
    \end{equation}
  Observing that
  \begin{align*}
    r(0) &= 1 & r'(0) &= \frac{1}{c} \\
    M(0) &= 1 & M'(0) &= \frac{\beta}{c}
  \end{align*}
  and
  \begin{align*}
    v(0) &= \|f\|_{L^1(\rho_s)} \\
    v'(0) &= w(0) = -\int_G
    Ef\,\rho_s\,dm + \frac{1}{c} \int_G f \log f\,\rho_s\,dm
  \end{align*}
  we see that \eqref{alpha-prime} reads
  \begin{align*}
    0 &\ge -\frac{\beta}{c} \|f\|_{L^1(\rho_s)} - \int_G Ef\,\rho_s\,dm
    \\ &\quad + \frac{1}{c} \int_G f\log f\,\rho_s\,dm - \frac{1}{c}
    \|f\|_{L^1(\rho_s)} \log \|f\|_{L^1(\rho_s)} 
  \end{align*}
  which after rearranging is precisely \eqref{sLSI}.
\end{proof}

\begin{remark}
  Theorem \ref{sHC-implies-sLSI} does not rely on any properties of
  LSH functions, except the assumption that they satisfy \eqref{sHC}.
  So more broadly, any appropriate class of functions satisfying
  \eqref{sHC} will also satisfy \eqref{sLSI}.
\end{remark}

\subsection{sLSI implies sHC}

In this section, we show that if the strong logarithmic Sobolev
inequality is satisfied for LSH functions, then so is strong
hypercontractivity.

We begin by noting that the semigroup $e^{-tE}$ is contractive on
log-subharmonic functions.  We assume some integrability on $|\nabla
f|$ but this assumption will be removed later.

\begin{lemma}\label{contractive}
  Suppose $f \in LSH \cap W^{1, 1+}(\rho_s)$.  Then for any $t \ge 0$
  we have
  \begin{equation*}
    \|e^{-tE} f\|_{L^1(\rho_s)} \le \|f\|_{L^1(\rho_s)}.
  \end{equation*}
\end{lemma}

\begin{proof}
  Let $T > 0$ be arbitrary.  Applying Lemma \ref{key-diff} with
  $r(t) \equiv 1 = q$, we have
  \begin{equation*}
    \frac{d}{dt} \|e^{-tE} f\|_{L^1(\rho_s)} = - \int_G E e^{-tE} f\,\rho_s
    \,dm.
  \end{equation*}
  Now from Lemma \ref{ui-lemma}, again with $r(t) \equiv
  1 = q$, we have in particular that $e^{-tE} f, |\nabla e^{-tE} f|, E
  e^{-tE} f \in L^1(\rho_s)$.  Also, $f$ is subharmonic and hence so
  is $e^{-tE} f$ by \eqref{Delta-dilate}.  So by Lemma \ref{time-space}, we have
  \begin{equation*}
    \int_G E e^{-tE} f\,\rho_s
    \,dm = \frac{s}{2} \int_G \Delta e^{-tE} f\,\rho_s\,dm \ge 0.
  \end{equation*}
  Hence $\|e^{-tE} f\|_{L^1(\rho_s)}$ is a decreasing function of $t$.
\end{proof}

The next step is to show that \eqref{sLSI} implies that \eqref{sHC}
holds at time $t=t_J$.  We take $p=1$ and again assume, for now,
sufficient integrability for $|\nabla f|$.

\begin{lemma}\label{sHC-tJ}
  Suppose that \eqref{sLSI} holds.  Let $1 \le q < \infty$, and
  set
  \begin{align*}
   t_J &= t_J(1,q) = c \log q \\
   M &= M(1,q) = \exp(\beta \cdot (1 - q^{-1})).
  \end{align*}
  Suppose $f \in LSH \cap W^{1,q+}(\rho_s)$.  Then
  \begin{equation}\label{slsi-at-tJ}
    \|e^{-t_J E} f\|_{L^q(\rho_s)} \le M \|f\|_{L^1(\rho_s)}.
  \end{equation}
\end{lemma}

\begin{proof}
  Set $r(t) = e^{t/c}$, so that $r(t_J) = q$, and let $f_t, v(t),
  w(t)$ be as in Lemma \ref{key-diff}.  The hypotheses of Lemma
  \ref{key-diff} are satisfied, with $T = t_J$, so we have $v \in
  C^1([0, t_J])$ and $v'(t) = w(t)$.

  On the other hand, for each $0 \le t \le t_J$, we have $f_t \in
  LSH$, by Proposition \ref{lsh-properties} items \ref{lsh-power} and
  \ref{lsh-dilate}.  Moreover, Lemma \ref{ui-lemma} implies $f_t \in
  W^{1,1+}(\rho_s)$.  So \eqref{sLSI} applies to $f_t$.  In terms of
  $v(t), w(t)$, this reads
  \begin{equation}\label{slsi-v-w}
    c w(t) \le v(t) \log v(t) + \beta v(t)
  \end{equation}
  where we note that $\frac{r'(t)}{r(t)} = \frac{1}{c}$.  
  Since $w(t) = v'(t)$, we may rewrite \eqref{slsi-v-w} as
  \begin{equation}\label{ddt-log-vt}
    \frac{d}{dt} \log v(t) \le \frac{1}{c} \log v(t) +
    \frac{\beta}{c}.
  \end{equation}
  Define
  \begin{align*}
    M(t) &= M(1, r(t)) = \exp(\beta \cdot (1 - e^{-t/c})) \\
    \alpha(t) &= \frac{1}{M(t)} \|e^{-tE} f\|_{L^{r(t)}(\rho_s)} =
        \frac{1}{M(t)} v(t)^{1/r(t)}
  \end{align*}
  as in the proof of Theorem \ref{sHC-implies-sLSI}. Note that
  $\alpha(0) = \|f\|_{L^1(\rho_s)}$ and $\alpha(t_J) = M^{-1}
  \|e^{-t_J E} f\|_{L^q(\rho_s)}$.  Then we have
  \begin{align*}
    \frac{d}{dt} \log \alpha(t) &= e^{-t/c} \left(-\frac{1}{c} \log v(t) +
    \frac{d}{dt} \log v(t) - \frac{\beta}{c} \right)  \le 0
  \end{align*}
  using \eqref{ddt-log-vt}.  Hence $\alpha(t)$ is a decreasing
  function on $[0,t_J]$, so in particular $\alpha(t_J) \le \alpha(0)$,
  which is the desired statement.
\end{proof}

\begin{theorem}\label{sLSI-implies-sHC}
  In any stratified Lie group $G$, if \eqref{sLSI} holds, then
  \eqref{sHC} holds, with the same constants $c,\beta$.
\end{theorem}

\begin{proof}
  Fix $f \in LSH \cap L^q(\rho_s)$ and $t \ge t_J(p,q)$.
  
  Suppose first that $p=1$ and $f \in LSH \cap W^{1,q+}(\rho_s)$.
  Then Lemma \ref{sHC-tJ} gives
  \begin{equation} \label{use-tJ}
    \|e^{-t_J E} f\|_{L^q(\rho_s)} \le M(1,q) \|f\|_{L^1(\rho_s)}.
  \end{equation}

  Set $\tau = t - t_J(1,q)$, and let $g =
  e^{-t_J E} f^q \in LSH$.  We apply Lemma \ref{ui-lemma} with $T =
  t_J$ and $r(t) \equiv q$, so that $g = f_{t_J}$, to see that $g \in
  W^{1, 1+}(\rho_s)$.  Applying Lemma \ref{contractive} to
  $g$, we have
  \begin{equation*}
    \|e^{-\tau E} g\|_{L^1(\rho_s)} \le \|g\|_{L^1(\rho_s)}
  \end{equation*}
  or in other words
  \begin{equation}\label{use-contractive}
    \|e^{-tE} f\|_{L^q(\rho_s)}^q \le \|e^{-t_J E} f\|_{L^q(\rho_s)}.
  \end{equation}
  Combining \eqref{use-tJ} and \eqref{use-contractive} gives
  \eqref{sHC} in this case.

  Next, suppose only that $p=1$, $f \in LSH \cap L^{q+}(\rho_s)$, but
  make no assumptions about $\nabla f$.  Let $\varphi_n$ be a sequence
  of standard mollifiers as in Lemma \ref{convolution-converge}, and
  set $f_n = \varphi_n \ast f$, so that $f_n \to f$ pointwise and in
  $L^{1+}(\rho_s)$; then $e^{-tE} f_n \to e^{-tE} f$ pointwise as
  well.  We have $f_n \in LSH$ by Lemma \ref{convolution-LSH}; $f_n
  \in L^{q+}(\rho_s)$ by Lemma \ref{convolution-Lp+}; and $|\nabla
  f_n| \in L^{q+}(\rho_s)$ by Lemma \ref{convolution-deriv}.  So by
  the previous case, we have
  \begin{equation}
    \|e^{-tE} f_n\|_{L^q(\rho_s)} \le M(1,q) \|f_n\|_{L^1(\rho_s)}
  \end{equation}
  and by Fatou's lemma, the same holds for $f$.

  Next, suppose $p=1$ and $f \in LSH \cap L^q(\rho_s)$.  Then for any
  $0 < \alpha < 1$, we have $f^{\alpha} \in LSH \cap L^{q+}(\rho_s)$,
  so that by the previous case,
  \begin{equation*}
    \|e^{-tE} f^\alpha\|_{L^q(\rho_s)} \le M(1,q) \|f^\alpha\|_{L^1(\rho_s)}.
  \end{equation*}
  Letting $\alpha \to 1$, we have $f^\alpha \to f$ pointwise and in
  $L^1(\rho_s)$ (by dominated convergence, using for instance $1+f$ as
  the dominating function).  So by Fatou's lemma, the result holds for
  $f$.

  Finally, let $0 < p \le q$ be arbitrary and $f \in LSH \cap
  L^q(\rho_s)$.  Set $g = f^p$ and $r = q/p$.  Then we have $g \in LSH
  \cap L^r(\rho_s)$, and so by the previous case we have
  \begin{equation*}
    \|e^{-tE} g\|_{L^r(\rho_s)} \le M(1,r) \|g\|_{L^1(\rho_s)}, \qquad
    t \ge t_J(1,r).
  \end{equation*}
  Noting that $M(1,r) = M(p,q)^p$ and $t_J(1,r) = t_J(p,q)$, this
  reads
  \begin{equation*}
    \|e^{-tE} f\|_{L^q(\rho_s)}^p \le M(p,q)^p \|f\|_{L^p(\rho_s)}^p
  \end{equation*}
  which is equivalent to \eqref{sLSI}.
\end{proof}

\section{Weaker notions of subharmonicity}\label{weak-lsh}

We have chosen to focus our attention in this paper on log-subharmonic
functions which are $C^2$.  In this section, we note that our results
for strong hypercontractivity can be extended to functions which are
LSH in a weaker sense.

A comprehensive discussion of the various possible definitions of
subharmonicity on stratified Lie groups is beyond the scope of this
paper.  We refer the reader to \cite{blu-book}, in which the basic
definition of subharmonic functions (Definition 7.2.2) is in terms of
harmonic measure.  Many other equivalent characterizations are given;
perhaps the simplest is the following definition in terms of
distributional derivatives.

\begin{definition}
  We say a function $f : G \to [-\infty, \infty)$ \ignorethis{(]} is
  \textbf{weakly subharmonic} if $f \in L^1_{\loc}(G,m)$ and $\Delta f
  \ge 0$ in the sense of distributions.  We say a function $f : G \to
      [0,\infty)$ \ignorethis{(]} is \textbf{weakly log-subharmonic}
      if either $f \equiv 0$ or $\log f$ is weakly subharmonic, and we
      write $f \in wLSH$.
\end{definition}

Strictly speaking, a function $f$ is weakly subharmonic in this sense
iff it has an $m$-version which is subharmonic in the sense of
\cite[Definition 7.2.2]{blu-book}; see \cite[Theorem 8.2.15 and
  Corollary 8.2.4]{blu-book}.  The distinction is irrelevant for our
current purposes, since null sets will not concern us.

Let us also call attention to \cite[Corollary 8.2.3]{blu-book}, where
it is shown that a function is (weakly) subharmonic iff it satisfies a
sub-averaging property, which is analogous to the definition of
subharmonic used in \cite{gkl2010,gkl2015}.

\begin{lemma}\label{wLSH-approx}
  Suppose $f \in wLSH \cap L^{q+}(\rho_s)$, where $q \ge 1$.  Then
  there is a sequence $f_n \in LSH \cap L^{q+}(\rho_s)$ with $f_n
  \to f$ almost everywhere and in $L^{1+}(\rho_s)$.
\end{lemma}

\begin{proof}
  If $f \equiv 0$ this is trivial by taking $f_n = 1/n$.  Otherwise,
  $\log f$ is weakly subharmonic.  Let $\varphi_n \in C_c^\infty$ be a
  sequence of nonnegative standard mollifiers with $\int_G
  \varphi_n\,dm = 1$, as in Lemma
  \ref{convolution-converge}, and set $g_n = \varphi_n \ast \log f$.
  Then $g_n \to \log f$ almost everywhere.  
  By \cite[Theorem 8.1.5 and Corollary 8.2.3]{blu-book}, $g_n$ is also weakly
  subharmonic; moreover, since $g_n \in C^\infty(G)$, we have by
  \cite[Proposition 7.2.5]{blu-book} that $\Delta g_n \ge 0$.

  Set $f_n = \exp(g_n)$, so that $f_n \in LSH$ and
  $f_n \to f$ almost everywhere.  Now $f_n \le \varphi_n \ast f$ by
  Jensen's inequality.  By Lemma \ref{convolution-Lp+}, we have
  $\varphi_n \ast f \in L^{q+}(\rho_s)$, so the same is true for
  $f_n$.  And by Lemma \ref{convolution-converge}, we have $\varphi_n
  \ast f \to f$ in $L^{1+}(\rho_s)$, so $f_n \to f$ in
  $L^{1+}(\rho_s)$ as well.
\end{proof}

\begin{theorem}
  If \eqref{sHC} holds for all $f \in LSH \cap L^q(\rho_s)$, then it
  holds for all $f \in wLSH \cap L^q(\rho_s)$.
\end{theorem}

\begin{proof}
  As in the proof of Theorem \ref{sLSI-implies-sHC}, it suffices to
  prove \eqref{sHC} with $p=1$ and for all $f \in wLSH \cap
  L^{q+}(\rho_s)$.  Using Lemma \ref{wLSH-approx}, choose $f_n \in LSH
  \cap L^{q+}(\rho_s)$ with $f_n \to f$ almost everywhere and in
  $L^{1+}(\rho_s)$.  Then \eqref{sHC} holds for each $f_n$.  We have
  $\|f_n\|_{L^{1}(\rho_s)} \to \|f\|_{L^1(\rho_s)}$, so by Fatou's
  lemma, \eqref{sHC} holds for $f$.
\end{proof}

In the setting of complex stratified Lie groups (Example
\ref{complex-example}), the modulus of a holomorphic function is
weakly LSH.

\begin{proposition}\label{holomorphic-wlsh}
  Let $G$ be a complex stratified Lie group, and suppose $f : G \to
  \mathbb{C}$ is holomorphic.  Then $|f| \in wLSH$.
\end{proposition}

\begin{proof}
  Let $f_n = \sqrt{|f|^2+\frac{1}{n}}$.  We showed in Lemma
  \ref{holomorphic-lsh} that $f_n \in LSH$.  Now $f_n \downarrow |f|$,
  and so $\log |f|$ is a decreasing limit of subharmonic functions.
  By \cite[Theorem 8.2.7]{blu-book}, $\log |f|$ is therefore weakly
  subharmonic.
\end{proof}

\section{Acknowledgments}

I would like to thank Bruce K.~Driver for suggesting this
problem and drawing attention to the papers \cite{gkl2010,gkl2015}
whose results are extended here.  I would also like to thank Leonard
Gross and Laurent Saloff-Coste for their collaboration on the paper
\cite{eldredge-gross-saloff-coste-dila}, of which this paper is an
outgrowth.  This research was supported by a grant from the
Simons Foundation (\#355659, Nathaniel Eldredge).

\bibliographystyle{plainnat}
\bibliography{allpapers}

\def\cprime{$'$} \def\cprime{$'$} \def\cprime{$'$} \def\cprime{$'$}
\begin{thebibliography}{40}
\providecommand{\natexlab}[1]{#1}
\providecommand{\url}[1]{\texttt{#1}}
\expandafter\ifx\csname urlstyle\endcsname\relax
  \providecommand{\doi}[1]{doi: #1}\else
  \providecommand{\doi}{doi: \begingroup \urlstyle{rm}\Url}\fi

\bibitem[Bakry(1997)]{bakry-sobolev}
D.~Bakry.
\newblock On {S}obolev and logarithmic {S}obolev inequalities for {M}arkov
  semigroups.
\newblock In \emph{New trends in stochastic analysis ({C}haringworth, 1994)},
  pages 43--75. World Sci. Publ., River Edge, NJ, 1997.

\bibitem[Bakry and \'Emery(1984)]{bakry-emery-short}
Dominique Bakry and Michel \'Emery.
\newblock Hypercontractivit\'e de semi-groupes de diffusion.
\newblock \emph{C. R. Acad. Sci. Paris S\'er. I Math.}, 299\penalty0
  (15):\penalty0 775--778, 1984.
\newblock ISSN 0249-6291.

\bibitem[Bakry et~al.(2008)Bakry, Baudoin, Bonnefont, and
  Chafa{\"{\i}}]{bbbc-jfa}
Dominique Bakry, Fabrice Baudoin, Michel Bonnefont, and Djalil Chafa{\"{\i}}.
\newblock On gradient bounds for the heat kernel on the {H}eisenberg group.
\newblock \emph{J. Funct. Anal.}, 255\penalty0 (8):\penalty0 1905--1938, 2008.
\newblock ISSN 0022-1236.
\newblock \doi{10.1016/j.jfa.2008.09.002}.
\newblock URL \url{http://dx.doi.org/10.1016/j.jfa.2008.09.002}.

\bibitem[Bonfiglioli et~al.(2007)Bonfiglioli, Lanconelli, and
  Uguzzoni]{blu-book}
A.~Bonfiglioli, E.~Lanconelli, and F.~Uguzzoni.
\newblock \emph{Stratified {L}ie groups and potential theory for their
  sub-{L}aplacians}.
\newblock Springer Monographs in Mathematics. Springer, Berlin, 2007.
\newblock ISBN 978-3-540-71896-3; 3-540-71896-6.

\bibitem[Carlen(1991)]{carlen-integral-identities}
Eric~A. Carlen.
\newblock Some integral identities and inequalities for entire functions and
  their application to the coherent state transform.
\newblock \emph{J. Funct. Anal.}, 97\penalty0 (1):\penalty0 231--249, 1991.
\newblock ISSN 0022-1236.
\newblock \doi{10.1016/0022-1236(91)90022-W}.
\newblock URL \url{http://dx.doi.org/10.1016/0022-1236(91)90022-W}.

\bibitem[Driver and Gross(1997)]{driver-gross-hilbert-spaces-holomorphic}
Bruce~K. Driver and Leonard Gross.
\newblock Hilbert spaces of holomorphic functions on complex {L}ie groups.
\newblock In \emph{New trends in stochastic analysis ({C}haringworth, 1994)},
  pages 76--106. World Sci. Publ., River Edge, NJ, 1997.

\bibitem[Driver and Melcher(2005)]{driver-melcher}
Bruce~K. Driver and Tai Melcher.
\newblock Hypoelliptic heat kernel inequalities on the {H}eisenberg group.
\newblock \emph{J. Funct. Anal.}, 221\penalty0 (2):\penalty0 340--365, 2005.
\newblock ISSN 0022-1236.
\newblock \doi{10.1016/j.jfa.2004.06.012}.
\newblock URL \url{http://dx.doi.org/10.1016/j.jfa.2004.06.012}.

\bibitem[Eldredge(2009)]{eldredge-precise-estimates}
Nathaniel Eldredge.
\newblock Precise estimates for the subelliptic heat kernel on {$H$}-type
  groups.
\newblock \emph{J. Math. Pures Appl. (9)}, 92\penalty0 (1):\penalty0 52--85,
  2009.
\newblock ISSN 0021-7824.
\newblock \doi{10.1016/j.matpur.2009.04.011}.
\newblock URL \url{http://dx.doi.org/10.1016/j.matpur.2009.04.011}.
\newblock arXiv:0810.3218.

\bibitem[Eldredge(2010)]{eldredge-gradient}
Nathaniel Eldredge.
\newblock Gradient estimates for the subelliptic heat kernel on {$H$}-type
  groups.
\newblock \emph{J. Funct. Anal.}, 258\penalty0 (2):\penalty0 504--533, 2010.
\newblock ISSN 0022-1236.
\newblock \doi{10.1016/j.jfa.2009.08.012}.
\newblock URL \url{http://dx.doi.org/10.1016/j.jfa.2009.08.012}.
\newblock arXiv:0904.1781.

\bibitem[Eldredge(2014)]{eldredge-complex-h-type}
Nathaniel Eldredge.
\newblock On complex {H}-type {L}ie algebras.
\newblock Preprint. arXiv:1406.2396, 2014.
\newblock URL \url{http://arxiv.org/abs/1406.2396}.

\bibitem[Eldredge et~al.(2015)Eldredge, Gross, and
  Saloff-Coste]{eldredge-gross-saloff-coste-dila}
Nathaniel Eldredge, Leonard Gross, and Laurent Saloff-Coste.
\newblock Strong hypercontractivity and logarithmic {S}obolev inequalities on
  stratified complex {L}ie groups.
\newblock To appear in Transactions of the American Mathematical Society.
  arXiv:1510.05151, 2015.
\newblock URL \url{http://arxiv.org/abs/1510.05151}.

\bibitem[Federbush(1969)]{federbush69}
Paul Federbush.
\newblock Partially alternate derivation of a result of {N}elson.
\newblock \emph{Journal of Mathematical Physics}, 10\penalty0 (1):\penalty0
  50--52, 1969.
\newblock \doi{10.1063/1.1664760}.
\newblock URL \url{http://dx.doi.org/10.1063/1.1664760}.

\bibitem[Glimm(1968)]{glimm68}
James Glimm.
\newblock Boson fields with nonlinear self-interaction in two dimensions.
\newblock \emph{Comm. Math. Phys.}, 8:\penalty0 12--25, 1968.
\newblock ISSN 0010-3616.

\bibitem[Graczyk et~al.(2010)Graczyk, Kemp, and Loeb]{gkl2010}
Piotr Graczyk, Todd Kemp, and Jean-Jacques Loeb.
\newblock Hypercontractivity for log-subharmonic functions.
\newblock \emph{J. Funct. Anal.}, 258\penalty0 (6):\penalty0 1785--1805, 2010.
\newblock ISSN 0022-1236.
\newblock \doi{10.1016/j.jfa.2009.08.014}.
\newblock URL \url{http://dx.doi.org/10.1016/j.jfa.2009.08.014}.

\bibitem[Graczyk et~al.(2015)Graczyk, Kemp, and Loeb]{gkl2015}
Piotr Graczyk, Todd Kemp, and Jean-Jacques Loeb.
\newblock Strong logarithmic {S}obolev inequalities for log-subharmonic
  functions.
\newblock \emph{Canad. J. Math.}, 67\penalty0 (6):\penalty0 1384--1410, 2015.
\newblock ISSN 0008-414X.
\newblock \doi{10.4153/CJM-2015-015-8}.
\newblock URL \url{http://dx.doi.org/10.4153/CJM-2015-015-8}.

\bibitem[Gross(1975)]{gross75}
Leonard Gross.
\newblock Logarithmic {S}obolev inequalities.
\newblock \emph{Amer. J. Math.}, 97\penalty0 (4):\penalty0 1061--1083, 1975.
\newblock ISSN 0002-9327.
\newblock \doi{10.2307/2373688}.
\newblock URL \url{http://dx.doi.org/10.2307/2373688}.

\bibitem[Gross(1999)]{gross-hypercontractivity-complex}
Leonard Gross.
\newblock Hypercontractivity over complex manifolds.
\newblock \emph{Acta Math.}, 182\penalty0 (2):\penalty0 159--206, 1999.
\newblock ISSN 0001-5962.
\newblock \doi{10.1007/BF02392573}.
\newblock URL \url{http://dx.doi.org/10.1007/BF02392573}.

\bibitem[Gross(2002)]{gross-strong-hypercontractivity}
Leonard Gross.
\newblock Strong hypercontractivity and relative subharmonicity.
\newblock \emph{J. Funct. Anal.}, 190\penalty0 (1):\penalty0 38--92, 2002.
\newblock ISSN 0022-1236.
\newblock \doi{10.1006/jfan.2001.3883}.
\newblock URL \url{http://dx.doi.org/10.1006/jfan.2001.3883}.
\newblock Special issue dedicated to the memory of I. E. Segal.

\bibitem[Gross(2006)]{gross-lsi-survey-2006}
Leonard Gross.
\newblock Hypercontractivity, logarithmic {S}obolev inequalities, and
  applications: a survey of surveys.
\newblock In \emph{Diffusion, quantum theory, and radically elementary
  mathematics}, volume~47 of \emph{Math. Notes}, pages 45--73. Princeton Univ.
  Press, Princeton, NJ, 2006.

\bibitem[H{\"o}rmander(1967)]{hormander67}
Lars H{\"o}rmander.
\newblock Hypoelliptic second order differential equations.
\newblock \emph{Acta Math.}, 119:\penalty0 147--171, 1967.
\newblock ISSN 0001-5962.
\newblock \doi{10.1007/BF02392081}.
\newblock URL \url{http://dx.doi.org/10.1007/BF02392081}.

\bibitem[Hu and Li(2010)]{li-gradient-h-type}
Jun-Qi Hu and Hong-Quan Li.
\newblock Gradient estimates for the heat semigroup on {H}-type groups.
\newblock \emph{Potential Anal.}, 33\penalty0 (4):\penalty0 355--386, 2010.
\newblock ISSN 0926-2601.
\newblock \doi{10.1007/s11118-010-9173-1}.
\newblock URL \url{http://dx.doi.org/10.1007/s11118-010-9173-1}.

\bibitem[Janson(1983)]{janson-hypercontractivity-1983}
Svante Janson.
\newblock On hypercontractivity for multipliers on orthogonal polynomials.
\newblock \emph{Ark. Mat.}, 21\penalty0 (1):\penalty0 97--110, 1983.
\newblock ISSN 0004-2080.
\newblock \doi{10.1007/BF02384302}.
\newblock URL \url{http://dx.doi.org/10.1007/BF02384302}.

\bibitem[Janson(1997)]{janson-complex-hypercontractivity-1997}
Svante Janson.
\newblock On complex hypercontractivity.
\newblock \emph{J. Funct. Anal.}, 151\penalty0 (1):\penalty0 270--280, 1997.
\newblock ISSN 0022-1236.
\newblock \doi{10.1006/jfan.1997.3144}.
\newblock URL \url{http://dx.doi.org/10.1006/jfan.1997.3144}.

\bibitem[Kaplan(1980)]{kaplan80}
Aroldo Kaplan.
\newblock Fundamental solutions for a class of hypoelliptic {PDE} generated by
  composition of quadratic forms.
\newblock \emph{Trans. Amer. Math. Soc.}, 258\penalty0 (1):\penalty0 147--153,
  1980.
\newblock ISSN 0002-9947.
\newblock \doi{10.2307/1998286}.
\newblock URL \url{http://dx.doi.org/10.2307/1998286}.

\bibitem[Li(2006)]{li-jfa}
Hong-Quan Li.
\newblock Estimation optimale du gradient du semi-groupe de la chaleur sur le
  groupe de {H}eisenberg.
\newblock \emph{J. Funct. Anal.}, 236\penalty0 (2):\penalty0 369--394, 2006.
\newblock ISSN 0022-1236.
\newblock \doi{10.1016/j.jfa.2006.02.016}.
\newblock URL \url{http://dx.doi.org/10.1016/j.jfa.2006.02.016}.

\bibitem[Li(2007)]{li-heatkernel}
Hong-Quan Li.
\newblock Estimations asymptotiques du noyau de la chaleur sur les groupes de
  {H}eisenberg.
\newblock \emph{C. R. Math. Acad. Sci. Paris}, 344\penalty0 (8):\penalty0
  497--502, 2007.
\newblock ISSN 1631-073X.
\newblock \doi{10.1016/j.crma.2007.02.015}.
\newblock URL \url{http://dx.doi.org/10.1016/j.crma.2007.02.015}.

\bibitem[Li(2010)]{li-heat-h-type}
Hong-Quan Li.
\newblock Estimations optimales du noyau de la chaleur sur les groupes de type
  {H}eisenberg.
\newblock \emph{J. Reine Angew. Math.}, 646:\penalty0 195--233, 2010.
\newblock ISSN 0075-4102.
\newblock \doi{10.1515/CRELLE.2010.070}.
\newblock URL \url{http://dx.doi.org/10.1515/CRELLE.2010.070}.

\bibitem[Lust-Piquard(2010)]{lust-piquard-ornstein-uhlenbeck}
Fran{\c{c}}oise Lust-Piquard.
\newblock Ornstein-{U}hlenbeck semi-groups on stratified groups.
\newblock \emph{J. Funct. Anal.}, 258\penalty0 (6):\penalty0 1883--1908, 2010.
\newblock ISSN 0022-1236.
\newblock \doi{10.1016/j.jfa.2009.11.012}.
\newblock URL \url{http://dx.doi.org/10.1016/j.jfa.2009.11.012}.

\bibitem[Montgomery(2002)]{montgomery}
Richard Montgomery.
\newblock \emph{A tour of subriemannian geometries, their geodesics and
  applications}, volume~91 of \emph{Mathematical Surveys and Monographs}.
\newblock American Mathematical Society, Providence, RI, 2002.
\newblock ISBN 0-8218-1391-9.

\bibitem[Nagel et~al.(1985)Nagel, Stein, and Wainger]{nagel-stein-wainger}
Alexander Nagel, Elias~M. Stein, and Stephen Wainger.
\newblock Balls and metrics defined by vector fields. {I}. {B}asic properties.
\newblock \emph{Acta Math.}, 155\penalty0 (1-2):\penalty0 103--147, 1985.
\newblock ISSN 0001-5962.
\newblock \doi{10.1007/BF02392539}.
\newblock URL \url{http://dx.doi.org/10.1007/BF02392539}.

\bibitem[Nelson(1966)]{nelson66}
Edward Nelson.
\newblock A quartic interaction in two dimensions.
\newblock In \emph{Mathematical {T}heory of {E}lementary {P}articles ({P}roc.
  {C}onf., {D}edham, {M}ass., 1965)}, pages 69--73. M.I.T. Press, Cambridge,
  Mass., 1966.

\bibitem[Nelson(1973)]{nelson-free-markov}
Edward Nelson.
\newblock The free {M}arkoff field.
\newblock \emph{J. Functional Analysis}, 12:\penalty0 211--227, 1973.

\bibitem[Rifford(2014)]{rifford}
Ludovic Rifford.
\newblock \emph{Sub-{R}iemannian geometry and optimal transport}.
\newblock SpringerBriefs in Mathematics. Springer, Cham, 2014.
\newblock ISBN 978-3-319-04803-1; 978-3-319-04804-8.
\newblock \doi{10.1007/978-3-319-04804-8}.
\newblock URL \url{http://dx.doi.org/10.1007/978-3-319-04804-8}.

\bibitem[Robinson(1991)]{robinson-book}
Derek~W. Robinson.
\newblock \emph{Elliptic operators and {L}ie groups}.
\newblock Oxford Mathematical Monographs. The Clarendon Press Oxford University
  Press, New York, 1991.
\newblock ISBN 0-19-853591-0.
\newblock Oxford Science Publications.

\bibitem[Stam(1959)]{stam59}
A.~J. Stam.
\newblock Some inequalities satisfied by the quantities of information of
  {F}isher and {S}hannon.
\newblock \emph{Information and Control}, 2:\penalty0 101--112, 1959.
\newblock ISSN 0890-5401.

\bibitem[Strichartz(1986)]{strichartz}
Robert~S. Strichartz.
\newblock Sub-{R}iemannian geometry.
\newblock \emph{J. Differential Geom.}, 24\penalty0 (2):\penalty0 221--263,
  1986.
\newblock ISSN 0022-040X.
\newblock URL \url{http://projecteuclid.org/euclid.jdg/1214440436}.

\bibitem[Strichartz(1989)]{strichartz-corrections}
Robert~S. Strichartz.
\newblock Corrections to: ``{S}ub-{R}iemannian geometry'' [{J}. {D}ifferential
  {G}eom.\ {\bf 24} (1986), no.\ 2, 221--263; {MR}0862049 (88b:53055)].
\newblock \emph{J. Differential Geom.}, 30\penalty0 (2):\penalty0 595--596,
  1989.
\newblock ISSN 0022-040X.
\newblock URL \url{http://projecteuclid.org/euclid.jdg/1214443604}.

\bibitem[Varopoulos(1990)]{varopoulos-II}
N.~Th. Varopoulos.
\newblock Small time {G}aussian estimates of heat diffusion kernels. {II}.
  {T}he theory of large deviations.
\newblock \emph{J. Funct. Anal.}, 93\penalty0 (1):\penalty0 1--33, 1990.
\newblock ISSN 0022-1236.
\newblock \doi{10.1016/0022-1236(90)90136-9}.
\newblock URL \url{http://dx.doi.org/10.1016/0022-1236(90)90136-9}.

\bibitem[Varopoulos et~al.(1992)Varopoulos, Saloff-Coste, and
  Coulhon]{purplebook}
N.~Th. Varopoulos, L.~Saloff-Coste, and T.~Coulhon.
\newblock \emph{Analysis and geometry on groups}, volume 100 of \emph{Cambridge
  Tracts in Mathematics}.
\newblock Cambridge University Press, Cambridge, 1992.
\newblock ISBN 0-521-35382-3.

\bibitem[Zhou(1991)]{zhou-contractivity-1991}
Zheng-Fang Zhou.
\newblock The contractivity of the free {H}amiltonian semigroup in the {$L\sb
  p$} space of entire functions.
\newblock \emph{J. Funct. Anal.}, 96\penalty0 (2):\penalty0 407--425, 1991.
\newblock ISSN 0022-1236.
\newblock \doi{10.1016/0022-1236(91)90067-F}.
\newblock URL \url{http://dx.doi.org/10.1016/0022-1236(91)90067-F}.

\end{thebibliography}

\end{document}